\documentclass[12pt]{article}
\usepackage{}

\usepackage{authblk}
\usepackage{cite}
\usepackage{amsmath}
\usepackage{amsthm,amsfonts,amssymb}           
\usepackage{bbm}
\usepackage[pdfstartview=FitH,
            CJKbookmarks=true,
            bookmarksnumbered=true,
            bookmarksopen=true,
            colorlinks,
            linkcolor=blue,
            anchorcolor=blue,
            citecolor=blue,
            urlcolor=blue
            ]{hyperref}
\numberwithin{equation}{section}

\newtheorem*{thm1.2'}{Theorem 1.2$'$}

\newtheorem{thm}{Theorem}[section]
\newtheorem{cor}[thm]{Corollary}
\newtheorem{lem}[thm]{Lemma}

\theoremstyle{remark}

\newtheorem{ex}[thm]{Example}
\def\R{{\mathbb R}}
\def\Q{{\mathbb Q}}
\def\lin{\operatorname{lin}}
\def\bd{\operatorname{bd}}
\def\relint{\operatorname{relint}}
\def\intr{\operatorname{int}}
\newcommand{\GL}[1]{\mathrm{GL}({#1})}                  
\newcommand{\SL}[1]{\mathrm{SL}({#1})}                  
\newcommand{\GLn}{\mathrm{GL}(n)}                  
\newcommand{\SLn}{\mathrm{SL}(n)}                  
\newcommand{\Sn}{S^{n-1}}                  
\def\Rn{\mathbb{R}^n}
\newcommand{\kik}[1]{\mathbbm{1}_{#1}}                  
\def\MP{\mathcal{P}}
\def\MK{\mathcal{K}}
\def\MA{\mathcal{A}}
\def\MN{\mathcal{N}}
\def\MT{\mathcal{T}}
\def\MKoon{\MK_{(o)}^n}
\def\MPoon{\MP_{(o)}^n}
\def\MKon{\MK_{o}^n}
\def\MPon{\MP_{o}^n}
\def\MTon{\MT_o^n}
\def\e{\varepsilon}
\def\Conv{\operatorname{Conv}(\R^n)}
\def\supp{\operatorname{supp}}

\makeatletter
\newcommand{\subjclass}[2][1991]{%
  \let\@oldtitle\@title%
  \gdef\@title{\@oldtitle\footnotetext{#1 \emph{Mathematics subject classification.} #2}}%
}
\newcommand{\keywords}[1]{%
  \let\@@oldtitle\@title%
  \gdef\@title{\@@oldtitle\footnotetext{\emph{Key words and phrases.} #1.}}%
}
\makeatother

\title{\bf{Affine function valued valuations}}
\author{Jin Li}
\affil{Institut f\"{u}r Diskrete Mathematik und Geometrie, Technische Universit\"{a}t Wien, 1040 Wien, Austria\\\href{mailto: Jin Li<lijin2955@gmail.com>}{lijin2955@gmail.com}}
\date{}
\subjclass[2010]{52A20, 52A39, 52B45}
\keywords{Valuation, $\SLn$ contravariant, general projection function, Orlicz projection function, $L_p$ projection function, mixed volume}

\begin{document}

\maketitle

\begin{abstract}
A classification of $\SLn$ contravariant, continuous function valued valuations on convex bodies is established.
Such valuations are natural extensions of $\SLn$ contravariant $L_p$ Minkowski valuations, the classification of which characterized $L_p$ projection bodies, which are fundamental in the $L_p$ Brunn-Minkowski theory, for $p \geq 1$.
Hence our result will help to better understand extensions of the $L_p$ Brunn-Minkowski theory.
In fact, our results characterize general projection functions which extend $L_p$ projection functions ($p$-th powers of the support functions of $L_p$ projection bodies) to projection functions in the $L_p$ Brunn-Minkowski theory for $0< p < 1$ and in the Orlicz Brunn-Minkowski theory.
\end{abstract}

\section{Introduction}\label{s1}
Let $\MK^n$ be the set of \emph{convex bodies} (i.e., compact convex set) in Euclidean space $\R^n$.
A \emph{valuation} is a map $Z$ from $\MK^n$ to an abelian semigroup $\langle \MA, +\rangle$ such that
\begin{align}\label{val100}
Z K + Z L = Z (K \cup L) + Z(K \cap L)
\end{align}
whenever $K,L, K \cup L \in \MK^n$. A map defined on a subset of $\MK^n$ is also called a valuation if \eqref{val100} holds whenever $K,L,K \cup L,K \cap L$ are contained in the subset. A \emph{function valued valuation} is a valuation taking values in some function space where addition in \eqref{val100} is the ordinary addition of functions.

Since any convex body (star body) can be identified with its support function (radial function), valuations taking values in the space of convex bodies (star bodies) are often studied as valuations taking values in some function space; see \cite{Lud02b,Lud05,Hab12b,SW2015mink,abardia2011p,Lud03,SS06,Sch2010,SW12,Lud12,Tsa12,Wan11,Par14a,Par14b,CLM2017Min,BL2017Min,Hab09,HL06,Lud06}. Function valued valuations are also an important tool for establishing results on other valuations, for example, measured valued valuations \cite{HP14a}.
When Ludwig \cite{Lud02b,Lud05}, Schuster, Wannerer \cite{SW12}, Haberl \cite{Hab12b} and Parapatits \cite{Par14a} studied $\SLn$ contravariant $L_p$ Minkowski valuations, they also gave classifications of $\SLn$ contravariant valuations taking values in some special function space. Here an $L_p$ Minkowski valuation is a valuation taking values in $\MK^n$ where addition in \eqref{val100} is $L_p$ Minkowski addition.

Let $p \geq 0$. A function $f:\R^n \to \R$ is called \emph{homogeneous of degree $p$} if $f(\lambda x) = \lambda^p f(x)$ for any $\lambda >0$ and $x \in \R^n$.
Let $\mathcal{C}(\R^n)$ be the set of continuous function on $\R^n$ and $\mathcal{C}_p(\R^n)$ be the subset of $\mathcal{C}(\R^n)$ such that any $f \in \mathcal{C}_p(\R^n)$ is homogeneous of degree $p$.

A function valued valuation $Z : \MK^n \to \mathcal{C}(\R^n)$ is called \emph{$\SLn$ contravariant} if
$$Z (\phi K) (x) = Z K (\phi ^{-1} x)$$
for every $K \in \mathcal{K} ^n$ and $\phi \in \SLn$. Let $h_K$ be the support function of $K$. For $p \geq 1$, an $L_p$ Minkowski valuation $Z$ is $\SLn$ contravariant if the map $K \mapsto h_{ZK}$ is an $\SLn$ contravariant function valued valuations.

Let $p \geq 1$. Due to Haberl \cite{Hab12b} and Parapatits \cite{Par14a}, roughly speaking, the set of $\SLn$ contravariant $L_p$ Minkowski valuations is the cone of asymmetric $L_p$ projection bodies, which were introduced by Ludwig \cite{Lud05},
and $\mathcal{C}_p(\R^n)$ valued valuations are the linear hull of asymmetric $L_p$ projection functions ($p$-th powers of the support functions of asymmetric $L_p$ projection bodies).
In this sense, $\SLn$ contravariant $L_p$ Minkowski valuations and $\mathcal{C}_p(\R^n)$ valued valuations are ``basically" the same.
In the dual case, where $\SLn$ contravariance is replaced by $\SLn$ covariance,
there are also no further $\mathcal{C}_p(\mathbb{R}^n)$ valued valuations than the $p$-th powers of the support functions of the corresponding $L_p$ Minkowski valuations;
see Haberl \cite{Hab12b}, Parapatits \cite{Par14b}, Li and Leng \cite{LL2016LpMV}.
However, if we remove the homogeneity assumption, then Laplace transforms of convex bodies (that is, classical Laplace transforms of indicator functions of convex bodies) are additional $\SLn$ covariant $\mathcal{C}(\R^n)$ valued valuations; see Li and Ma \cite{LM2017Lap}.
Hence the natural question arises to give a unified classification of $\SLn$ contravariant and of $\SLn$ covariant $\mathcal{C}(\R^n)$ valued valuations.
We believe such a classification will help to better understand extensions of the $L_p$ Brunn-Minkowski theory.

In this paper, we give a classification of $\SLn$ contravariant $\mathcal{C}(\R^n)$ valued valuations for dimension $n \geq 3$. The cases $n=1$ and $n=2$ are rather different and will be treated separately. Hence we will always assume $n \geq 3$ throughout the paper.

The topology of $\MK^n$ is induced by the Hausdorff metric and the topology of $\mathcal{C}(\mathbb{R}^n)$ is the $C^{0}$ topology induced by uniform convergence on any compact subset. If we identify $\MK^n$ as the cone of support functions, then the topology of $\MK^n$ induced by the Hausdorff metric is the same as the $C^{0}$ topology of the cone of support functions. With this topology, we can define continuity and (Borel) measurability of maps from $\MK^n$ to $\mathcal{C}(\R^n)$.

Let $\mathcal{K}_o ^n$ be the set of convex bodies in $\R^n$ containing the origin.
\begin{thm}\label{thm:MKo}
Let $n \ge 3$. A map $Z : \mathcal{K}_o ^n \to \mathcal{C}(\mathbb{R}^n)$ is a continuous, $\SLn$ contravariant valuation
if and only if there are constants $c_0,c_{n-1} \in \R$ and a continuous function $\zeta : \R \to \R$ satisfying $\lim_{|t|\to \infty}\zeta(t)/t =0 $ such that
\begin{align*}
Z K (x) &= \int_{\Sn \setminus \{h_K=0\}} \zeta \left(\frac{x \cdot u}{h_{K}(u)}\right)d V_K(u) + c_{n-1} V_1(K,[-x,x])  + c_0V_0(K)
\end{align*}
for every $K \in \mathcal {K}_o^n$ and $x \in \R^n$. Moreover, $c_0,c_{n-1}$ and $\zeta$ are uniquely determined by $Z$.
\end{thm}
Here $V_1(K,[-x,x]) = h_{\Pi K}(x)$ is the classical projection function of $K$, where $\Pi K$ is the projection body of $K$, $V_0(K)$ is the Euler characteristic, $\{h_K=0\}$ denotes the set $\{u \in \Sn : h_K(u)=0 \}$ for $K \in \MK^n$ and $V_K$ is the \emph{cone-volume measure} of $K$. Using the surface area measure $S_K$, the cone-volume measure can be written as $dV_K = \frac{1}{n} h_K dS_K$.
If $K=\{o\}$, then $\int_{\Sn \setminus \{h_K=0\}} \zeta \left(\frac{x \cdot u}{h_{K}(u)}\right)d V_K(u) =0$.
See Section \ref{s2} for details.

Let $\mathcal{P}_o ^n$ be the set of polytopes in $\R^n$ containing the origin.
We can replace continuity by measurability when considering valuations on polytopes.
Here Borel sets in the space of polytopes are also induced by the Hausdorff metric.

\begin{thm}\label{thm:cont}
Let $n \ge 3$. A map $Z : \mathcal{P}_o ^n \to \mathcal{C}(\mathbb{R}^n)$ is a measurable, $\SLn$ contravariant valuation if and only if
there are constants $c_0,c_0',c_{n-1} \in \R$ and a continuous function $\zeta : \R \to \R$ such that
\begin{align}\label{val22}
Z P (x)&=\int_{\Sn \setminus \{h_P=0\}} \zeta \left(\frac{x \cdot u}{h_{P}(u)}\right)d V_P(u) + c_{n-1} V_1(P,[-x,x])  \notag \\
&\qquad + c_0V_0(P) + c_0' (-1)^{\dim P} \mathbbm{1}_{\relint P}(o)
\end{align}
for every $P \in \mathcal {P}_o^n$ and $x \in \R^n$. Moreover, $c_0,c_0',c_{n-1}$ and $\zeta$ are uniquely determined by $Z$.
\end{thm}
Here $\dim P$ is the dimension of the affine hull of $P$ and $\relint P$ is the relative (with respect to the affine hull of $P$) interior of $P$. The function $\mathbbm{1}_{L} (o)$ is the indicator function of the set $L \subset \R^n$ at the origin $o$, that is, if $o \in L$, then $\mathbbm{1}_{L} (o)=1$, otherwise $\mathbbm{1}_{L} (o)=0$.

If we further assume that $Z K \in \mathcal{C}_p(\R^n)$ for $p \geq 1$ in Theorems \ref{thm:MKo} and \ref{thm:cont}, then we get classification results of Haberl \cite{Hab12b} and Parapatits \cite{Par14a}; see Corollary \ref{p>=1}.
The classification of the corresponding $L_p$ Minkowski valuations is a direct corollary by further assuming that $Z K$ is the $p$-th power of a support function.
Our results also give valuations associated with the $L_p$ Brunn-Minkowski theory for $0 <p < 1$ and the Orlicz theory; namely, the function $Z_\zeta K (x) := \int_{\Sn \setminus \{h_K=0\}} \zeta \left(\frac{x \cdot u}{h_{K}(u)}\right)d V_K(u)$ is (an extension of) the $L_p$ and Orlicz projection functions
depending on the choice of the function $\zeta: \R \to \R$; see Corollaries \ref{p<1} and \ref{co:Orl}.
We leave the details to Section \ref{s2}.

\emph{Real valued valuations} are valuations taking values in $\R$ with scalar addition and $Z: \MK^n \to \R$ is $\SLn$ invariant if $Z(\phi K) = Z K$ for any $\phi \in \SLn$ and $K \in \MK_o^n$. Theorem \ref{thm:MKo} (Theorem \ref{thm:cont}) also imply the classification of $\SLn$ invariant, continuous (measurable) real valued valuations which were obtained before by Blaschke \cite{Blas1937vor} and Ludwig and Reitzner \cite{LR2017sl}.
This follows from the fact that any $\SLn$ invariant real valued valuation can be understood as an $\SLn$ contravariant function valued valuation taking values in constant functions.
More precisely, if $\zeta \equiv c$, then $\int_{\Sn \setminus \{h_K=0\}} \zeta \left(\frac{x \cdot u}{h_{K}(u)}\right)d V_K(u) = cV_n(K)$, where $V_n(K)$ is the $n$ dimensional volume of $K$.
A specific characterization of all $\SLn$ invariant real valued valuations is also established in Corollary \ref{p<1} for the case $p=0$.

Considering valuations themselves are homogeneous, a different characterization of $L_p$ projection functions and of all $\SLn$ invariant real valued valuations is also established in Corollary \ref{qhom}.

Since general ($L_p$ or Orlicz) projection functions of the convex body $K$ are a special case of the general ($L_p$ or Orlicz) first mixed volumes of a convex body and a segment, our result might be a first step towards characterizing general first mixed volumes in valuation theory.
In particular, Corollaries \ref{trans} and \ref{trans2} are related to the characterization of classical mixed volumes by Alesker and Schuster \cite{AS2017+}. We also give a characterization of $L_p$ first mixed volumes in Corollary \ref{co:mixv} for $p \geq 1$. Other special cases of classical mixed volumes are intrinsic volumes (mixed volume of a convex body and the unit ball). A celebrated characterization of intrinsic volumes is the Hadwiger theorem. It is the starting point of valuation theory; see also \cite{Kla95,Ale99,Ale01,BF2011herm,LR10,HP14b}.
For another characterization of classical mixed volumes (not in valuation theory), see Milman and Schneider \cite{MS11}. In the following, when we talk about mixed volumes, we will always refer to the first mixed volume.

A classification of $\SLn$ contravariant valuations on all convex bodies or on all polytopes that do not necessarily contain the origin can be established by Theorems \ref{thm:MKo} and \ref{thm:cont}.
Let $\MP^n$ be the set of polytopes in $\R^n$. For $K\in \MK^n$, let $[K,o]$ denote the convex hull of $K$ and the origin.
\begin{thm}\label{thm:MK}
Let $n \ge 3$. A map $Z : \MK^n \to \mathcal{C}(\mathbb{R}^n)$ is a continuous, $\SLn$ contravariant valuation if and only if
there are constants $c_0,c_{n-1},\widetilde{c}_{n-1} \in \R$ and continuous functions $\zeta, \widetilde{\zeta} : \R \to \R$ satisfying $\lim_{|t|\to \infty}\zeta(t)/t =0$ and $\lim_{|t|\to \infty}\widetilde{\zeta}(t)/t =0 $ such that
\begin{align*}
&Z K (x) \notag\\
&=\int_{\Sn \setminus \{h_K =0\}} \zeta \left(\frac{x \cdot u}{h_{K}(u)}\right)d V_K(u) + \int_{\Sn \setminus \{h_{[K,o]} =0\}} \widetilde{\zeta} \left(\frac{x \cdot u}{h_{[K,o]}(u)}\right)d V_{[K,o]}(u) \notag\\
&\qquad + c_{n-1}V_1(K,[-x,x]) +\widetilde{c}_{n-1} V_1([K,o],[-x,x])  + c_0V_0(K)
\end{align*}
for every $K \in \mathcal {K}_o^n$ and $x \in \R^n$. Moreover, $c_0,c_{n-1},\widetilde{c}_{n-1} $ and $\zeta, \widetilde{\zeta} $ are uniquely determined by $Z$.
\end{thm}
Here $dV_K = \frac{1}{n} h_K dS_K$ is a signed measure since $h_K$ might be negative.

\begin{thm}\label{thm:MP}
Let $n \ge 3$. A map $Z : \mathcal{P}^n \to \mathcal{C}(\mathbb{R}^n)$ is a measurable, $\SLn$ contravariant valuation if and only if
there are constants $c_0,c_0',\widetilde{c}_0,c_{n-1},\widetilde{c}_{n-1} \in \R$ and continuous functions $\zeta, \widetilde{\zeta} : \R \to \R$ such that
\begin{align}\label{val24}
&Z P (x)  \notag\\
&=\int_{\Sn \setminus \{h_P =0\}} \zeta \left(\frac{x \cdot u}{h_{P}(u)}\right)d V_P(u) + \int_{\Sn \setminus \{h_{[P,o]} =0\}} \widetilde{\zeta} \left(\frac{x \cdot u}{h_{[P,o]}(u)}\right)d V_{[P,o]}(u) \notag\\
&\qquad + c_{n-1}V_1(P,[-x,x]) +\widetilde{c}_{n-1} V_1([P,o],[-x,x])  \notag\\
&\qquad + c_0V_0(P)  + c_0' (-1)^{\dim P} \mathbbm{1}_{\relint P}(o) + \widetilde{c}_0 \mathbbm{1}_{P}(o)
\end{align}
for every $P \in \mathcal {P}_o^n$ and $x \in \R^n$. Moreover, $c_0,c_0',\widetilde{c}_0,c_{n-1},\widetilde{c}_{n-1}$ and $\zeta, \widetilde{\zeta} $ are uniquely determined by $Z$.
\end{thm}

All the theorems will be proved in Section \ref{sec:proofthm} and all the corollaries will be proved in Section \ref{sec:proofco}.

\section{$L_p$ and Orlicz projection functions and mixed volumes}\label{s2}
We refer to Schneider \cite{Sch14} as a general reference for convex geometry.

The \emph{support function} of a convex body $K$ is $h_K(x) = \max_{y \in K} \{x \cdot y \}$, $x \in \R^n$. It is easy to see that support functions are convex functions and homogeneous of degree $1$. Moreover, support functions are important tools in convex geometry because of the following fact: given a convex function $h: \R^n \to \R$ which is homogeneous of degree $1$, there exists a unique convex body such that $h=h_K$. Briefly, a convex body is identified with its support function.

The \emph{Hausdorff distance} of $K,L$ is $\max_{u \in \Sn} |h_K(u) -h_L(u)|$. Hence $K_i \to K$ with respect to the Hausdorff metric if and only if $h_{K_i} \to h_{K}$ uniformly on $\Sn$.

First, let $p \geq 1$. The \emph{$L_p$ Minkowski sum} of $K,L \in \MK_o^n$ introduced by Firey (generalizing the classical Minkowski sum) is defined by its support function
\begin{align}\label{padd}
h_{K+_p L} = (h_K^p+h_L^p)^{1/p}.
\end{align}
Let $\MK_{(o)}^n$ be the set of convex bodies in $\R^n$ containing the origin in their interiors. The \emph{$L_p$ mixed volume} of $K \in \MKoon$ and $L \in \MKon$ is
\begin{align}\label{pmix1}
V_p(K,L) := \lim_{\e \to 0^+} \frac{p}{n}\frac{V_n(K+_{p,\e} L)-V_n(K)}{\e} = \frac{1}{n}\int_{\Sn} h_L^p(u) h_{K}^{1-p}(u)d S_K(u),
\end{align}
where $K+_{p,\e} L$ is the $L_p$ combination of $K,L$ such that $h_{K+_{p, \e} L}^p = h_K^p+ \e h_L^p$ and $S_K$ is the \emph{surface area measure} of $K$, that is, the pushforward of the $(n-1)$-dimensional Lebesgue measure with respect to the Gauss map.

One important property of $L_p$ mixed volumes is that they satisfy the $L_p$ Minkowski inequality: for $K \in \MKoon$ and $L \in \MKon$
\begin{align}\label{pMi}
\left(\frac{V_p(K,L)}{V_n(K)}\right)^{\frac{1}{p}} \geq \left(\frac{V_n(L)}{V_n(K)}\right)^{\frac{1}{n}}.
\end{align}
The $L_p$ Minkowski inequality is equivalent to the $L_p$ Brunn-Minkowski inequality.
The classical Minkowski inequality for $p=1$ is due to Minkowski himself.
For $p >1$, the $L_p$ Minkowski inequality was first established by Lutwak \cite{Lut93} and is the starting point of the systematic study of the $L_p$ Brunn-Minkowski theory.
If $p=1$ and $L$ is the unit ball, then the $L_p$ Minkowski inequality implies the classical isoperimetric inequality.
Moreover, the $L_p$ Minkowski inequality and its equality conditions are critical to many problems, for example, $L_p$ Minkowski problems \cite{Lut93}.

The $L_p$ mixed volume of $K \in \MKoon$ and a segment $[-x,x]$ is
\begin{align}\label{pmix2}
V_p(K,[-x,x]) = \frac{1}{n}\int_{\Sn} |x \cdot u|^p h_{K}^{1-p}(u)d S_K(u).
\end{align}
When $p=1$ and $x \in \Sn$, the right side of \eqref{pmix2} is (up to a constant) the $(n-1)$-dimensional volume of $K|x^\bot$, where $K|x^\bot$ is the orthogonal projection of $K$ onto the hyperplane $x^\bot = \{y \in \Rn: y \cdot x =0\}$. The function $Z_p K (x) := V_p(K,[-x,x])$ is called the \emph{$L_p$ projection function} of $K$.
The classical projection function for $p=1$ was introduced by Minkowski and $L_p$ versions were introduced by Lutwak, Yang and Zhang \cite{LYZ00a}. There is an important affine inequality associated with $L_p$ projection functions, namely, the $L_p$ Petty projection inequality \cite{Pet71,LYZ00a}: for $K \in \MKoon$
\begin{align*}
V_n(K)^\frac{n-p}{p} \int_{\Sn} V_p(K,[-x,x])^{-\frac{n}{p}}dx \leq V_n(E)^\frac{n-p}{p} \int_{\Sn} V_p(E,[-x,x])^{-\frac{n}{p}}dx,
\end{align*}
where $E$ is an ellipsoid.
By the Jensen inequality, the classical Petty projection inequality is also stronger that the classical isoperimetric inequality. Unfortunately, there is (so far) no clear relationship between the $L_p$ Minkowski inequality and the $L_p$ Petty projection inequality. Haberl and Schuster \cite{HS09a} established $L_p$ Petty projection inequalities for the asymmetric $L_p$ projection functions, i.e., linear combinations of $V_p(K,[o,x])$ and $V_p(K,[o,-x])$. The functional version of the $L_p$ Petty projection inequality is the affine $L_p$ Sobolev inequality; see \cite{Zha99Sob,LYZ02a,HS09b,Wang13PS}. The reverse classical Petty projection inequality is the Zhang projection inequality \cite{Zhang91pro}.

In the classical case $p=1$, the mixed volume and the projection function can be defined for any convex body. All the above still holds. We still write
$$V_1(K,[-x,x]):= \frac{1}{n}\int_{\Sn} |x \cdot u| d S_K(u)$$
for $K \in \MK^n$. The asymmetric case is the same since $V_1(K,[-x,x]) =2 V_1(K,[o,\pm x])$. But there are other extensions of $L_p$ projection functions onto $\MPon$, namely
\begin{align*}
\widehat{V}_p(P,[-x,x]):= \frac{1}{n}\int_{\Sn \setminus \{h_P=0\}} |x \cdot u|^p h_P^{1-p}(u)d S_P(u).
\end{align*}
Also, the asymmetric $L_p$ projection functions are defined as
\begin{align*}
\widehat{V}_p(P,[o,\pm x]) :&= \frac{1}{n}\int_{\Sn \setminus \{h_P=0\}} (x \cdot u)_\pm^p h_P^{1-p}(u)d S_P(u)\\
&= \int_{\Sn \setminus \{h_P=0\}} \left(\frac{x \cdot u}{h_{P}(u)}\right)_\pm^p d V_P(u)
\end{align*}
where $(\cdot)_\pm = \max \{\pm (\cdot) ,0\}$. They are both function valued valuations. \text{Clearly} $\widehat{V}_p(P,[-x,x]) = \widehat{V}_p(P,[o,x]) + \widehat{V}_p(P,[o,- x])$.
Moreover, $V_1(K,[-x,x]) \\= h_{\Pi K} (x)$ and $\widehat{V}_p(P,[o,\pm x]) = h_{\widehat{\Pi}_p^\pm P}^p(x)$.
Here $\Pi$ is the classical projection body and $\widehat{\Pi}_p^\pm $ are the asymmetric $L_p$ projection bodies.
If we assume that $Z P \in \mathcal{C}_p(\R^n)$ for $p \geq 1$ in Theorem \ref{thm:cont}, then we obtain those function valued valuations which were already characterized before by Haberl \cite{Hab12b} and Parapatits \cite{Par14a}.
\begin{cor}[Haberl \cite{Hab12b} and Parapatits \cite{Par14a}]\label{p>=1}
A map $Z : \mathcal{P}_o ^n \to \mathcal{C}_1(\mathbb{R}^n)$ is a measurable, $\SLn$ contravariant valuation if and only if there exist constants
$c_{n-1},\hat{c}_{n-1}^+,\hat{c}_{n-1}^- \in \mathbb{R}$ such that
\begin{align*}
Z P (x) &=c_{n-1} V_1(P,[-x,x]) + \hat{c}_{n-1}^+ \widehat{V}_1(P,[o,x]) + \hat{c}_{n-1}^- \widehat{V}_1(P,[o, -x])
\end{align*}
for every $P \in \mathcal{P}_o ^n$ and $x \in \R^n$.

For $1 < p < \infty$, a map $Z : \mathcal{P}_o ^n \to \mathcal{C}_p(\mathbb{R}^n)$ is a measurable, $\SLn$ contravariant valuation if and only if there exist constants
$\hat{c}_{n-p}^+,\hat{c}_{n-p}^- \in \R$ such that
\begin{align*}
Z P (x) &=\hat{c}_{n-p}^+ \widehat{V}_p(P,[o,x]) + \hat{c}_{n-p}^- \widehat{V}_p(P,[o, -x])
\end{align*}
for every $P \in \mathcal{P}_o ^n$ and $x \in \R^n$.
\end{cor}

%

There are two different ways to extend the $L_p$ Brunn-Minkowski theory.
One is to $p<1$. For $0<p<1$, the right side of \eqref{padd} is in general not a support function.
Now the support function of $K+_p L$ is defined as the maximum support function smaller than $(h_K^p+h_L^p)^{1/p}$. Then, \eqref{pmix1} still holds and defines the $L_p$ mixed volume.
Also \eqref{pmix2} still gives $L_p$ projection functions for $0<p<1$. B\"{o}r\"{o}czky, Lutwak, Yang and Zhang \cite{BLYZ12} established the $L_p$ Minkowski inequality \eqref{pMi} for $0<p<1$ for planar origin-symmetric convex bodies and conjectured that it also holds for $n$ dimensional origin-symmetric convex bodies.
The $L_p$ Petty projection inequality for $0<p<1$ is unknown. For other aspects of the $L_p$ Brunn-Minkowski theory for $0 \leq p <1$, see \cite{BLYZ13,BLYZ14,CLM2017BM,CLZ2017lpMA,Zhu2014log}

We extend the $L_p$ projection functions for $0< p <1$ to $\MKon$ as follows
\begin{align*}
V_p(K,[o,\pm x]) :&= \frac{1}{n}\int_{\Sn} (x \cdot u)_\pm^p h_K^{1-p}(u)d S_K(u).
\end{align*}
Here we write $V_p$ instead of $\widehat{V}_p$ since $\int_{\Sn \setminus \{h_K=0\}} (x \cdot u)_\pm^p h_K^{1-p}(u)d S_K(u) = \int_{\Sn} (x \cdot u)_\pm^p h_K^{1-p}(u)d S_K(u)$ for $0<p<1$.
For valuations associated with the $L_p$ Brunn-Minkowski theory for $0 <p < 1$, we get the following by Theorem \ref{thm:cont}.

\begin{cor}\label{p<1}
For $0<p<1$, a map $Z : \mathcal{P}_o ^n \to \mathcal{C}_p(\mathbb{R}^n)$ is a measurable, $\SLn$ contravariant valuation if and only if
there are constants $\hat{c}_{n-p}^+,\hat{c}_{n-p}^- \in \R$ such that
\begin{align*}
Z P (x)= \hat{c}_{n-p}^+ V_p(P,[o,x]) + \hat{c}_{n-p}^- V_p(P,[o,-x])
\end{align*}
for every $P \in \mathcal {P}_o^n$ and $x \in \R^n$. A map $Z : \mathcal{P}_o ^n \to \mathcal{C}_0(\mathbb{R}^n)$ is a measurable, $\SLn$ contravariant valuation if and only if
there are constants $c_0,c_0',c_n \in \R$ such that
\begin{align*}
Z P (x)= c_nV_n(P) + c_0V_0(P) + c_0' (-1)^{\dim P} \mathbbm{1}_{\relint P}(o)
\end{align*}
for every $P \in \mathcal {P}_o^n$ and $x \in \R^n$.
\end{cor}

For $0<p<1$, Haberl and Parapatits \cite{HP14a} obtained the corresponding result for even valuations. We remark that continuous versions of Corollaries \ref{p>=1} and \ref{p<1} are easy to get and non-zero continuous valuations on $\MKon$ only exists for $0 \leq p \leq 1$.

\textbf{Continuous version of Corollary 2.2:} \emph{For $0<p<1$, a map $Z : \mathcal{K}_o ^n \to \mathcal{C}_p(\mathbb{R}^n)$ is a continuous, $\SLn$ contravariant valuation if and only if
there are constants $\hat{c}_{n-p}^+,\hat{c}_{n-p}^- \in \R$ such that
\begin{align*}
Z K (x)= \hat{c}_{n-p}^+ V_p(K,[o,x]) + \hat{c}_{n-p}^- V_p(K,[o,-x])
\end{align*}
for every $K \in \mathcal {K}_o^n$ and $x \in \R^n$. A map $Z : \mathcal{K}_o ^n \to \mathcal{C}_0(\mathbb{R}^n)$ is a continuous, $\SLn$ contravariant valuation if and only if
there are constants $c_0,c_n \in \R$ such that
\begin{align*}
Z K (x)= c_nV_n(K) + c_0V_0(K)
\end{align*}
for every $K \in \mathcal {K}_o^n$ and $x \in \R^n$.}

Another extension of the $L_p$ Brunn-Minkowski theory is the so called Orlicz Brunn-Minkowski theory. Let $\zeta : \R \to [0,\infty)$ be a convex function such that $\zeta(0)=0$. We define the \emph{Orlicz projection function} $Z_\zeta K$ by extending \eqref{pmix2} to
\begin{align}\label{Opf1}
Z_\zeta K(x) = \int_{\Sn} \zeta \left(\frac{x \cdot u}{h_{K}(u)}\right)d V_K(u), ~~x\in \R^n.
\end{align}
In general, $Z_\zeta K(x)$ is not a support function. To obtain a convex body, Lutwak, Yang and Zhang \cite{LYZ10a} introduce
\begin{align*}
h_{\Pi_\zeta K}(x) :=\min \left\{\lambda >0  : \int_{\Sn} \zeta \left(\frac{x \cdot u}{\lambda h_{K}(u)}\right)d V_K(u)\leq V_n(K) \right\}.
\end{align*}
Since the right side is a support function, this introduces a convex body $\Pi_\zeta K$, the so called Orlicz projection body.
An Orlicz Petty projection inequality was established in \cite{LYZ10a,Bor2013} and a functional version in \cite{Lin2017OPS}.
However, Li and Leng \cite{LL2017OV} showed that Orlicz projection bodies are not valuations in the following sense:
there is no non-trivial $\SLn$ contravariant, convex body valued valuation with respect to the non-associative Orlicz addition on $\MPon$ (Orlicz addition is associative if and only if it is $L_p$ addition for some $p \geq 1$).
Meanwhile, the Orlicz projection function defined by \eqref{Opf1} is also closely related to Orlicz addition,
which was introduced by Gardner, Hug and Weil \cite{GHW14} and Xi, Jin and Leng \cite{XJL14} as an extension of $L_p$ addition. Orlicz addition preserves continuity and is compatible with $\GLn$ transforms.
Let $\varphi: [0,\infty) \to [0,\infty)$ be a convex function satisfying $\varphi(0)=0$ and $\varphi(1)=1$.
The \emph{Orlicz combination}, $K+_{\varphi,\e} L$, of $K,L$ with respect to $\varphi$ is defined by
\begin{align*}
\varphi\left(\frac{h_K}{h_{K+_{\varphi,\e} L}}\right)+\e \varphi\left(\frac{h_L}{h_{K+_{\varphi,\e} L}}\right)=1.
\end{align*}
For $K \in \MK_{(o)}^n$ and $L \in \MK_o^n$, the \emph{Orlicz mixed volume}
\begin{align*}
V_\varphi(K,L):=\lim_{\e \to 0^+} \frac{\varphi_l'(1)}{n}\frac{V_n(K+_{\varphi,\e} L)-V_n(K)}{\e} = \int_{\Sn} \varphi\left(\frac{h_L(u)}{h_K(u)}\right) dV_K(u),
\end{align*}
where $\varphi_l'(1)$ is the left derivative of $\varphi$ at $1$.

The Orlicz Brunn-Minkowski inequality states that
\begin{align*}
\frac{V_\varphi (K,L)}{V_n(K)} \geq  \varphi \left(\left(\frac{V_n(L)}{V_n(K)}\right)^{1/n}\right).
\end{align*}
For other aspects of Orlicz Brunn-Minkowski theory, see \cite{HLYZ10,HuaH12,LR10,Lud10a,LYZ10b,ZHY017orl,WXL2017+}

Now the Orlicz projection function $Z_\zeta K(x)$ defined in \eqref{Opf1} can be written as
\begin{align*}
Z_\zeta K(x) = V_{\varphi_1}(K,[o,x]) + V_{\varphi_2}(K,[-x,o])
\end{align*}
for $\varphi_1, \varphi_2: [0,\infty) \to [0,\infty)$ such that $\varphi_1(t) = \zeta(t)$ and $\varphi_2(t) = \zeta(-t)$ if $\zeta(\pm 1) =1$. 

We use the same notation $Z_\zeta$ to denote the extension of \eqref{Opf1} to $\MK^n$ for a general continuous function $\zeta$,
\begin{align*}
Z_\zeta K (x):= \int_{\Sn \setminus \{h_K=0\}} \zeta \left(\frac{x \cdot u}{h_{K}(u)}\right)d V_K(u), ~~ x\in \R^n,
\end{align*}
when the integral is finite.
Theorems \ref{thm:MKo}-\ref{thm:MP} show that the extension is natural in valuation theory.

We call a valuation \emph{simple} if it vanishes on lower dimensional convex bodies. Let $\Conv$ denote the set of convex functions from $\R^n$ to $\R$. We obtain the following characterization of Orlicz projection functions.
\begin{cor}\label{co:Orl}
A map $Z : \mathcal{P}_o ^n \to \Conv$ is a measurable, simple and $\SLn$ contravariant valuation if and only if there exists a convex function $\zeta: \R \to \R$ such that
\begin{align*}
Z P (x) = \int_{\Sn \setminus \{h_P=0\}} \zeta \left(\frac{x \cdot u}{h_{P}(u)}\right)d V_P(u)
\end{align*}
for every $P \in \MPon$ and $x \in \R^n$. Moreover, the function $\zeta$ is uniquely determined by $Z$.
\end{cor}

\section{Further classification results}
Corollaries \ref{p>=1} and \ref{p<1} characterize $L_p$ projection functions as valuations taking values in functions which are homogeneous of degree $p$. We can also characterize $L_p$ projection functions as homogeneous valuations.
Here we say that a valuation $Z : \mathcal{P}_o ^n \to \mathcal{C}(\mathbb{R}^n)$ is \emph{homogeneous of degree $p$} if $Z(\lambda K) = \lambda^{p} Z K$ for $\lambda >0$. Set $\delta_p^i=1$ for $p=i$ and $\delta_p^i=0$ otherwise.
\begin{cor}\label{qhom}
A map $Z : \mathcal{P}_o ^n \to \mathcal{C}(\mathbb{R}^n)$ is an $\SLn$ contravariant valuation which is homogeneous of degree $n-p$ if and only if there exist constants $c_0,c_0',c_{n-1},\hat{c}_{n-p}^+,\hat{c}_{n-p}^-,c_n \in \R^n$ such that
\begin{align*}
&Z P (x) \\
&=\begin{cases}
\hat{c}_{n-p}^+ \widehat{V}_{p}(P,[o,x]) + \hat{c}_{n-p}^- \widehat{V}_{p}(P,[o, -x]) & \\
\qquad  + \delta_p^{1} c_{n-1} V_1(P,[-x,x]) & \\
\qquad  + \delta_p^n (c_0V_0(P) + c_0' (-1)^{\dim P} \mathbbm{1}_{\relint P}(o)), & p\geq 1, \\
\hat{c}_{n-p}^+ V_{p}(P,[o,x]) + \hat{c}_{n-p}^- V_{p}(P,[o, -x]), & 0<p<1, \\
c_nV_n(P), &p=0, \\
0, & p<0.
\end{cases}
\end{align*}
for every $P \in \mathcal{P}_o ^n$ and $x \in \R^n$.
\end{cor}

If we further assume \emph{translation invariance} ($Z (P+y) = Z P$ for every $P \in \MPon$ (or $\MP^n$) and $y \in \R^n$ such that $P+y \in \MPon$ (or $\MP^n$)), then we characterize the classical projection function, volume and the Euler characteristic. This is a special case of characterizing the classical mixed volumes.
\begin{cor}\label{trans}
A map $Z : \mathcal{P}_o ^n \to \mathcal{C}(\mathbb{R}^n)$ is a measurable, translation invariant and $\SLn$ contravariant valuation if and only if there exist constants $c_0,c_{n-1},c_n \in \R$ such that
\begin{align*}
Z P(x) =  c_n V_n(P) + c_{n-1} V_1(P,[-x,x]) + c_0 V_0 (P)
\end{align*}
for every $P \in \mathcal{P}_o ^n$ and $x \in \R^n$.
\end{cor}

\begin{cor}\label{trans2}
A map $Z : \MP^n \to \mathcal{C}(\mathbb{R}^n)$ is a measurable, translation invariant and $\SLn$ contravariant valuation if and only if there exist constants $c_0,c_{n-1},c_n \in \R$ such that
\begin{align*}
Z P(x) =  c_n V_n(P) + c_{n-1} V_1(P,[-x,x]) + c_0 V_0 (P)
\end{align*}
for every $P \in \mathcal{P}_o ^n$ and $x \in \R^n$.
\end{cor}

We omit other versions of all the Corollaries corresponding to Theorems \ref{thm:MKo}, \ref{thm:MK} and \ref{thm:MP} since they are similar and easy to establish.

Finally, we show how our results are related to the characterization of $L_p$ mixed volumes. Although the following corollary is not a strong result, we think it might be inspiring.

%

A map $Z : \MK^n \times \MK^n \to \R$ is called $\SLn$ invariant if $Z(\phi K, \phi L) = Z(K,L)$ for any $K,L \in \MK^n$ and $\phi \in \SLn$. We say that $Z$ is a valuation with respect to the first variable if $Z(\cdot,L)$ is a valuation for any fixed $L \in \MK^n$, and $L_p$ additive with respect to the second variable if $Z(K,L_1+_p L_2) = Z(K,L_1) + Z(K,L_2)$ for any $K \in \MK^n, L_1,L_2 \in \MK^n$. For $p>1$, we further assume that $L_1,L_2 \in \MKon$.
Let $\MK_c^n$ be the set of symmetric convex bodies in $\R^n$ centered at the origin.

\begin{cor}\label{co:mixv}
Let $p \geq 1$ and $p$ not an even integer. A map $Z : \MPon \times \MK_c^n \to \R$ is an $\SLn$ invariant map which is a measurable valuation with respect to the first variable and continuous, $L_p$ additive with respect to the second variable if and only if there exist constants $\widehat{c}_{n-p},c_{n-1} \in \R$ such that
\begin{align*}
Z (P,L) =  \widehat{c}_{n-p} \widehat{V}_p(P,L) + c_{n-1}\delta_p^1 V_1(P,L)
\end{align*}
for every $P \in \MPon$ and $L \in \MK_c^n$.
\end{cor}

\section{Valuations and $\SLn$ contravariance}\label{exa}
Let $[A_1,\dots,A_i]$ denote the convex hull of the sets $A_1,\dots,A_i$ in $\R^n$.

\begin{thm}\label{thm:val}
Let $\zeta: \R \to \R$ be a continuous function and define a map $Z_\zeta : \MK^n \to \mathcal{C}(\R^n)$ by
\begin{align*}
Z_\zeta K (x)= \int_{\Sn \setminus \{h_K=0\}} \zeta \left(\frac{x \cdot u}{h_{K}(u)}\right)d V_K(u), ~~ x\in \R^n,
\end{align*}
for every $K \in \MK^n$ if the integral exists and is finite for every $x \in \R^n$.
We have the following conclusions: \\
(i) If $Z_\zeta$ is well defined on $\MK^n$ (or $\MP^n$), then $Z_\zeta$ is an $\SLn$ contravariant valuation on $\MK^n$ (or $\MP^n$). \\
(ii) $Z_\zeta $ is well defined and measurable on $\MP^n$ without any restriction on $\zeta$. \\
(iii) If $\lim_{|t|\to \infty}\zeta(t)/t =0$, then $Z_\zeta$ is well defined and continuous on $\MK^n$.
\end{thm}
\begin{proof}
(i) First, we show that $Z_\zeta$ is $\SLn$ contravariant.
Let $\phi \in \SL{n}$, $K \in \MK^n$. For any Borel set $\omega \subset \Sn$, we have $V_{\phi K} (\omega) = V_K(\overline{\phi^t \omega})$, where $\overline{\phi^t \omega} = \{v=\frac{\phi^{t}u}{|\phi^{t}u|}: u \in \omega \}$; see \cite{BLYZ13}. Then
\begin{align*}
\int_{\Sn} f(u)dV_{\phi K}(u) = \int_{\Sn} f\left(\frac{\phi^{-t}v}{|\phi^{-t}v|}\right)dV_{K}(v)
\end{align*}
for any continuous function $f$ on $\Sn$. Also, since $\{h_{\phi K}=0\} = \overline{\phi^{-t} \{h_K=0\}}$,
\begin{align*}
Z_\zeta (\phi K) (x) &=\int_{\Sn \setminus \{h_{\phi K}=0\}} \zeta \left(\frac{x \cdot u}{h_{\phi K}(u)}\right)d V_{\phi K}(u) \\
&=\int_{\Sn \setminus \{h_K=0\}} \zeta \left(\frac{x \cdot \frac{\phi^{-t}v}{|\phi^{-t}v|}}{h_{\phi K}\left(\frac{\phi^{-t}v}{|\phi^{-t}v|}\right)}\right)d V_{K}(v) \\
&=\int_{\Sn \setminus \{h_K=0\}} \zeta \left(\frac{\phi^{-1} x \cdot v}{h_{K}(v)}\right)d V_{K}(v) \\
&=Z_\zeta K (\phi^{-1} x).
\end{align*}

Second, we show that $Z_\zeta$ is a valuation. Let $K,L \in \MK^n$ such that $K \cup L \in \MK^n$. We divide $\Sn$ into three parts
\begin{align*}
\omega_1 = \{u \in \Sn: h_K(u)=h_L(u)\}, \\
\omega_2 = \{u \in \Sn: h_K(u)>h_L(u)\}, \\
\omega_3 = \{u \in \Sn: h_K(u)<h_L(u)\}.
\end{align*}
Let $\nu_K^{-1}$ denote the \emph{reverse Gauss map}, that is, for any $u \in \Sn$, $\nu_K^{-1}(u)$ is the set of boundary point of $K$ such that $u$ is a normal vector corresponding to those points. We have
\begin{align*}
\nu_{K \cup L}^{-1} (A_1)&= \nu_{K}^{-1} (A_1) \cup \nu_{L}^{-1} (A_1),~ \nu_{K \cap L}^{-1} (A_1)= \nu_{K}^{-1} (A_1) \cap \nu_{L}^{-1} (A_1), \\
&\nu_{K \cup L}^{-1} (A_2) = \nu_{K}^{-1}(A_2),~ \nu_{K \cap L}^{-1} (A_2) = \nu_{L}^{-1}(A_2),\\
&\nu_{K \cup L}^{-1} (A_3) = \nu_{L}^{-1}(A_3),~ \nu_{K \cap L}^{-1} (A_3) = \nu_{K}^{-1}(A_3)
\end{align*}
for any Borel set $A_i \subset \omega_i$.
Also
\begin{align*}
h_{K \cup L}(u)= \max\{h_K(u),h_L(u)\},~
h_{K \cap L}(u)= \min\{h_K(u),h_L(u)\}.
\end{align*}
Recall that the surface area measure is the pushforward of the $(n-1)$-dimensional Lebesgue measure with respect to the Gauss map, we have
\begin{align*}
V_{K \cup L}(A_1) + V_{K \cap L}(A_1) = V_K(A_1) + V_L(A_1), \\
V_{K \cup L}(A_2) = V_K(A_2),~V_{K \cap L}(A_2) =V_L(A_2),\\
V_{K \cup L}(A_3) = V_L(A_3),~V_{K \cap L}(A_3) =V_K(A_3)
\end{align*}
for any Borel set $A_i \subset \omega_i$.
Thus
\begin{align*}
&\int_{\omega_1 \setminus \{h_{K \cup L}=0\}} \zeta \left(\frac{x \cdot u}{h_{K \cup L}(u)}\right)d V_{K\cup L}(u)
+ \int_{\omega_1 \setminus \{h_{K \cap L}=0\}} \zeta \left(\frac{x \cdot u}{h_{K \cap L}(u)}\right)d V_{K\cap L}(u) \\
&\qquad =\int_{\omega_1 \setminus \{h_{K}=h_{L}=0\}} \zeta \left(\frac{x \cdot u}{h_{K}(u)=h_{L}(u)}\right)(d V_{K\cup L}(u)+d V_{K\cap L}(u)) \\
&\qquad =\int_{\omega_1 \setminus \{h_K=0\}} \zeta \left(\frac{x \cdot u}{h_{K}(u)}\right)d V_K(u)
+ \int_{\omega_1 \setminus \{h_L=0\}} \zeta \left(\frac{x \cdot u}{h_{L}(u)}\right)d V_L(u),
\end{align*}
\begin{align*}
&\int_{\omega_2 \setminus \{h_{K \cup L}=0\}} \zeta \left(\frac{x \cdot u}{h_{K \cup L}(u)}\right)d V_{K\cup L}(u)
+ \int_{\omega_2 \setminus \{h_{K \cap L}=0\}} \zeta \left(\frac{x \cdot u}{h_{K \cap L}(u)}\right)d V_{K\cap L}(u) \\
&\qquad =\int_{\omega_2 \setminus \{h_{K}=0\}} \zeta \left(\frac{x \cdot u}{h_{K}(u)}\right) d V_{K\cup L}(u)+\int_{\omega_2 \setminus \{h_{L}=0\}} \zeta \left(\frac{x \cdot u}{h_{L}(u)}\right)d V_{K\cap L}(u) \\
&\qquad =\int_{\omega_2 \setminus \{h_K=0\}} \zeta \left(\frac{x \cdot u}{h_{K}(u)}\right)d V_K(u)
+ \int_{\omega_2 \setminus \{h_L=0\}} \zeta \left(\frac{x \cdot u}{h_{L}(u)}\right)d V_L(u)
\end{align*}
and
\begin{align*}
&\int_{\omega_3 \setminus \{h_{K \cup L}=0\}} \zeta \left(\frac{x \cdot u}{h_{K \cup L}(u)}\right)d V_{K\cup L}(u)
+ \int_{\omega_3 \setminus \{h_{K \cap L}=0\}} \zeta \left(\frac{x \cdot u}{h_{K \cap L}(u)}\right)d V_{K\cap L}(u) \\
&\qquad =\int_{\omega_3 \setminus \{h_{L}=0\}} \zeta \left(\frac{x \cdot u}{h_{L}(u)}\right) d V_{K\cup L}(u)+\int_{\omega_3 \setminus \{h_{K}=0\}} \zeta \left(\frac{x \cdot u}{h_{K}(u)}\right)d V_{K\cap L}(u) \\
&\qquad =\int_{\omega_3 \setminus \{h_L=0\}} \zeta \left(\frac{x \cdot u}{h_{L}(u)}\right)d V_L(u)
+ \int_{\omega_3 \setminus \{h_K=0\}} \zeta \left(\frac{x \cdot u}{h_{K}(u)}\right)d V_K(u).
\end{align*}
All together we get that $Z_\zeta$ is a valuation.

(ii) It is clear that $Z_\zeta$ is well defined on $\MP^n$.

To show that $\zeta$ is measurable on $\MP^n$, we can rewrite $Z_\zeta = Z_\zeta^+ - Z_\zeta^-$ with
\begin{align*}
Z_\zeta^\pm P (x) = \int_{\Sn \setminus \{h_P=0\}} \left(\zeta \left(\frac{x \cdot u}{h_{P}(u)}\right) h_P(u)\right)_{\pm}d S_P(u)
\end{align*}
for every $P \in \MP^n$ and $x \in \R^n$.
We claim that $Z_\zeta ^\pm (\cdot) (x)$ are lower semi-continuous on $\MP^n$ for every $x \in \R^n$.

Indeed, let $P_i,P \in \MP^n$ and $P_i \to P$.
First if $h_P(u) > 0$ for all $u \in \Sn$,
then for sufficiently large $i$, $h_{P_i}(u) > 0$ for all $u \in \Sn$.
Since $\left(\zeta \left(\frac{x \cdot u}{h_{P_i}(u)}\right)h_{P_i}(u)\right)_{+} \to \left(\zeta \left(\frac{x \cdot u}{h_{P}(u)}\right)h_{P}(u)\right)_{+}$ uniformly on any compact set $C \times \Sn \ni (x,u)$ and the surface area measures $S_{P_i} \to S_{P}$ weakly,
it is easy to see that $Z_\zeta P_i(x) \to Z_\zeta P(x)$.

Now assume that there is a $u \in \Sn$ such that $h_P(u) =0$. Since $P$ is a polytope, there is a suitable $\delta >0$ such that $S_P(\{0<|h_P| \leq \delta \})=0$. Here $\{0<|h_P| \leq \delta\}:=\{u \in \Sn: 0<|h_P(u)| \leq \delta\}$. We have
\begin{align*}
&\lim\limits_{i \to \infty} \int_{ \{|h_P| > \delta \} } \zeta \left(\left(\frac{x \cdot u}{h_{P_i}(u)}\right)h_{P_i}(u)\right)_+ dS_{P_i}(u) \\
& \qquad = \int_{ \{|h_P| > \delta \} } \left(\zeta \left(\frac{x \cdot u}{h_{P}(u)}\right)h_{P}(u)\right)_+ dS_P(u)
\end{align*}
uniformly on any compact set $C \ni x$ and
\begin{align*}
&\int_{ \{|h_P| \leq \delta \} \setminus \{h_{P_i}=0\} } \zeta \left(\left(\frac{x \cdot u}{h_{P_i}(u)}\right)h_{P_i}(u)\right)_+ dS_{P_i}(u)  \\
& \qquad \geq 0 \\
& \qquad = \int_{ \{|h_P| \leq \delta \} \setminus \{h_{P}=0\} } \zeta \left(\left(\frac{x \cdot u}{h_{P}(u)}\right)h_{P}(u)\right)_+ dS_{P}(u).
\end{align*}
Hence
$$\liminf\limits_{i \to \infty} Z_\zeta^+ P_i (x) \geq Z_\zeta^+ P (x).$$
Similarly $Z_\zeta^-(\cdot)(x)$ is lower semi-continuous.
Moreover, the lower semi-continuity is locally uniform with respect to $x$. That is to say, for any compact set $C \subset \R^n$ and $\varepsilon>0$, we have $Z_\zeta^\pm P_i (x) > Z_\zeta^\pm P (x) - \varepsilon$ for sufficient large $i$ not depending on the choice of $x \in C$.

Next we show that $Z_\zeta^\pm$ are measurable.
Let $\mathcal{S}(C,U):= \{g \in \mathcal{C}(\R^n): g(C) \subset U\}$, where $C$ is a compact set in $\R^n$ and $U$ is an open set in $\R$.
The collection of all $\mathcal{S}(C,U)$ forms a subbase of $\mathcal{C}(\R^n)$; see \cite[Section 46]{Mun2000topo}.
Let $\mathcal{B}_\Q$ denote the set of balls in $\R^n$ whose centers and radii are rational and let $\mathcal{U}_\Q$ denote the set of connected open sets in $\R$ whose end points are rational. The collection of $\mathcal{S}(C,U): C \in \mathcal{B}_\Q, U \in \mathcal{U}_\Q$ is a subbase without changing the topology. Hence $\mathcal{C}(\R^n)$ is a second countable topological space.
Recall that a topology space is second countable if it has a countable base. Also, a function is measurable if the preimage of every open set is a Borel set. Thus we only need to show that $(Z_\zeta^\pm)^{-1} (\mathcal{S}(C,U))$ is a Borel set in $\MP^n$ for every compact set $C \subset \R^n$ and connected open set $U \subset \R$. $U$ can be written as $(t_1,t_2)$, $(-\infty, t)$ and $(t,\infty)$ for $t,t_1,t_2 \in \R$. Also since
\begin{align*}
\mathcal{S}(C,(t_1,t_2)) = \mathcal{S}(C,(t_1,\infty)) \cap \mathcal{S}(C,(-\infty, t_2))
\end{align*}
and
\begin{align*}
\mathcal{S}(C,(-\infty, t))=\bigcup_{i=1}^\infty \mathcal{S}(C,(-\infty, t-1/i]),
\end{align*}
we only need to show that preimages of $\mathcal{S}(C,(-\infty, t])$ and $\mathcal{S}(C,(t,\infty))$ are Borel sets for all $t \in \R$. Let $P_i \in (Z_\zeta^\pm)^{-1} (\mathcal{S}(C,(-\infty,t]))$ such that $P_i \to P \in \MP^n$. For any $x \in C$, we have
\begin{align*}
Z_\zeta^\pm P (x) \leq \liminf\limits_{i \to \infty} Z_\zeta^\pm P_i (x) \leq t.
\end{align*}
Hence $P \in (Z_\zeta^\pm)^{-1} (\mathcal{S}(C,(-\infty,t]))$.
That is to say $(Z_\zeta^\pm)^{-1} (\mathcal{S}(C,(-\infty,t]))$ is closed in $\MP^n$.
Also, for any $P \in (Z_\zeta^\pm)^{-1} (\mathcal{S}(C,(t,\infty)))$, $\min_{x \in C} Z_\zeta^\pm P(x) >t$ since $Z_\zeta^\pm P \in \mathcal {C}(\R^n)$.
The fact that the lower semi-continuity of $Z_\zeta^\pm (\cdot) (x)$ is locally uniform implies that there is a neighborhood of $P$ such that for any $Q$ in this neighborhood, we have $Z_\zeta^\pm Q (C) >t$.
Hence $(Z_\zeta^\pm)^{-1} (\mathcal{S}(C,(t,\infty)))$ is open.
This proves that $Z_\zeta^\pm$ are measurable. Since the minus in $\mathcal{C}(\R^n)$ is continuous, the difference of two measurable function is measurable, which completes the proof of measurability.

(iii) Finally, let $K_i,K \in \MK^n$, $i=1,2 \dots$ such that $K_i \to K$.
We want to show that $Z_\zeta K$ is well defined and $Z_\zeta K_i \to Z_\zeta K $ uniformly on any compact set $C \subset \R^n$ if $\lim_{|t|\to \infty}\zeta(t)/t =0$.

If $h_K (u)> 0$ for all $u \in \Sn$, clearly $Z_\zeta K$ is well defined. The argument that $Z_\zeta K_i(x) \to Z_\zeta K(x)$ uniformly on the compact set $C$ in this case is similar to the part of measurability.

Now assume that there is a $u \in \Sn$ such that $h_K(u)=0$.
If $K \neq \{o\}$, then $\{h_K=0\}$ lies in a hemisphere.
Since $h_K$ and $\zeta$ are continuous and $\lim_{|t|\to \infty}\zeta(t)/t =0$, for any $\e >0$, there exists $c>0$ and $\e>\delta>0$ such that
\begin{align*}
\left|\zeta \left(\frac{x \cdot u}{h_{K}(u)}\right)\right| \leq \max \{\e \frac{|x \cdot u|}{|h_{K}(u)|}, c \}
\end{align*}
whenever $u \in \{|h_K| \leq \delta\} \setminus \{h_K=0\}$.
Also, since $h_{K_i} \to h_{K}$ uniformly on $\Sn$, for sufficiently large $i>0$, we have $\{h_{K_i}=0\} \subset \{|h_K| \leq \delta \}$ and
\begin{align*}
|h_{K_i}(u)| < \e, ~~
\left|\zeta \left(\frac{x \cdot u}{h_{K_i}(u)}\right)\right| \leq \max \{\e \frac{|x \cdot u|}{|h_{K_i}(u)|}, c\}
\end{align*}
whenever $u \in \{|h_K| \leq \delta \} \setminus \{h_{K_i}=0\}$.

Clearly $\int_{\{|h_K| > \delta \}} \zeta \left(\frac{x \cdot u}{h_{K}(u)}\right)dV_K(u)$ exists and is finite. Also
\begin{align*}
&\left|\int_{\{|h_K| \leq \delta \} \setminus \{h_K=0\}} \zeta \left(\frac{x \cdot u}{h_{K}(u)}\right)dV_K(u)\right| \\
& \qquad \leq \frac{1}{n}\int_{\{|h_K| \leq \delta \} \setminus \{h_K=0\}} \left|\zeta \left(\frac{x \cdot u}{h_{K}(u)}\right)\right| |h_{K}(u)| dS_K(u) \\
&\qquad \leq \frac{1}{n}\int_{\{|h_K| \leq \delta \} \setminus \{h_K=0\}} \max \{\e \frac{|x \cdot u|}{|h_{K}(u)|}, c \} |h_{K}(u)| dS_K(u) \\
&\qquad \leq \frac{1}{n} \max \{\e|x|S_K(\Sn), c \delta S_K(\Sn)\} \\
&\qquad \leq \frac{1}{n} \e\max \{|x|S_K(\Sn), c S_K(\Sn)\}
\end{align*}
These show that $Z_\zeta K$ is well defined.
Also, since $S_{K_i} \to S_{K}$ weakly, similarly we have
\begin{align*}
&\left|\int_{\{|h_K| \leq \delta \} \setminus \{h_{K_i}=0\}} \zeta \left(\frac{x \cdot u}{h_{K_i}(u)}\right)h_{K_i}(u) dS_{K_i}(u)\right| \\
&\qquad < \frac{1}{n}\e \max \{|x|S_{K_i}(\Sn), c S_{K_i}(\Sn)\} \\
&\qquad < \e\max \{|x|S_{K}(\Sn), c S_{K}(\Sn)\}
\end{align*}
for sufficiently large $i>0$. Furthermore, we have
\begin{align*}
&\left|\int_{ \{|h_K| > \delta \} } \zeta \left(\frac{x \cdot u}{h_{K_i}(u)}\right)h_{K_i}(u) dS_{K_i}(u) \right. \\
&\qquad \left. - \int_{ \{|h_K| > \delta \} }\zeta \left(\frac{x \cdot u}{h_{K}(u)}\right)h_{K}(u) dS_K(u)\right| \to 0
\end{align*}
uniformly on any compact set. Hence,
\begin{align*}
&\left|\int_{ \{h_{K_i} \neq 0 \}} \zeta \left(\frac{x \cdot u}{h_{K_i}(u)}\right)h_{K_i}(u) dS_{K_i}(u) - \int_{ \{h_K \neq 0 \}} \zeta \left(\frac{x \cdot u}{h_{K}(u)}\right)h_{K}(u) dS_K(u)\right| \\
& \leq \left|\int_{ \{|h_K| > \delta \}} \zeta \left(\frac{x \cdot u}{h_{K_i}(u)}\right)h_{K_i}(u) dS_{K_i}(u) - \int_{  \{|h_K| > \delta \}}\zeta \left(\frac{x \cdot u}{h_{K}(u)}\right)h_{K}(u) dS_K(u)\right| \\
&\qquad + \left|\int_{\{|h_K| \leq \delta \} \setminus \{h_K=0\}} \zeta \left(\frac{x \cdot u}{h_{K}(u)}\right)h_{K}(u) dS_K(u)\right| \\
&\qquad + \left|\int_{\{|h_K| \leq \delta \} \setminus \{h_{K_i}=0\}} \zeta \left(\frac{x \cdot u}{h_{K_i}(u)}\right)h_{K_i}(u) dS_{K_i}(u)\right| \\
& \to 0
\end{align*}
uniformly on any compact set.

If $K=\{o\}$, then $Z_\zeta K(x)=0$ and $\{h_K=0\} = \R^n$. With a similar argument we have
\begin{align*}
|Z_\zeta K_i(x)| \to 0.
\end{align*}
uniformly on any compact set, which completes the proof.
\end{proof}

\begin{cor}\label{co:MK}
Let $\zeta: \R \to \R$ be a continuous function and define a map $\widetilde{Z}_\zeta : \MK^n \to \mathcal{C}(\R^n)$ by
\begin{align*}
\widetilde{Z}_\zeta K (x)= \int_{\Sn \setminus \{h_{[K,o]}=0\}} \zeta \left(\frac{x \cdot u}{h_{[K,o]}(u)}\right)d V_{[K,o]}(u), ~~ x\in \R^n,
\end{align*}
for every $K \in \MK^n$ if the integral exists and is finite for every $x \in \R^n$.
We have the following conclusions: \\
(i) If $\widetilde{Z}_\zeta$ is well defined on $\MK^n$ (or $\MP^n$), then $\widetilde{Z}_\zeta$ is an $\SLn$ contravariant valuation on $\MK^n$ (or $\MP^n$). \\
(ii) $\widetilde{Z}_\zeta $ is well defined and measurable on $\MP^n$ without any restriction on $\zeta$. \\
(iii) If $\lim_{|t|\to \infty}\zeta(t)/t =0$, then $\widetilde{Z}_\zeta$ is well defined and continuous on $\MK^n$.
\end{cor}
\begin{proof}
Clearly $\widetilde{Z}_\zeta K = Z[K,o]$ for $Z$ defined in Theorem \ref{thm:val}. The $\SLn$ contravariance and valuation property follows from Theorem \ref{thm:val} and the fact that $[\phi K,o] = \phi[K,o]$
\begin{align*}
[K\cup L,o] = [K,o] \cup [L,o], ~~[K\cap L,o] = [K,o] \cap [L,o].
\end{align*}
when $K,L,K \cup L \in \MK^n$ and $\phi \in \SLn$.
Other statements also follow from Theorem \ref{thm:val} and the map $K \to [K,o]$ is continuous.
\end{proof}

The following examples are critical for lower dimensional convex bodies.
\begin{ex}\label{ex2}
The maps mapping $K \in \MK^n$ to $V_0(K)$, $V_0([K,o])$, $\mathbbm{1}_{K}(o)$ or $(-1)^{\dim K} \mathbbm{1}_{\relint K}(o)$
are function valued valuations taking values in constant functions. Moreover, they are $\SLn$ invariant and contravariant. $V_0(K)$ and $V_0([K,o])$ are continuous, while $\mathbbm{1}_{K}(o)$ and $(-1)^{\dim K} \mathbbm{1}_{\relint K}(o)$ are measurable but not continuous.
\end{ex}

\begin{ex}\label{ex3}
The maps mapping $K \in \MK^n$ to $h_{\Pi K}(x) = V_1(K,[-x,x])$ or $h_{\Pi [K,o]}(x) = V_1([K,o],[-x,x])$ are continuous, $\SLn$ contravariant function valued valuations.
Note that,
\begin{align}\label{pi1}
V_1(sT^{n-1},[-x,x]) = \frac{2s^{n-1}|x_n|}{n!} , ~~x \in \R^n
\end{align}
for $T^{n-1} = [o,e_1,\dots,e_{n-1}]$, where $x_n$ is the $n$-th coordinate of $x$.
\end{ex}

One direction of Theorems \ref{thm:MKo}-\ref{thm:MP} and Corollaries \ref{p>=1}-\ref{trans2} follows directly from Theorem \ref{thm:val}, Corollary \ref{co:MK}, Example \ref{ex2} and Example \ref{ex3}. In the following, we only need to prove the other direction that valuations satisfying all conditions have the corresponding representations.

A function valued valuation $Z$ on $\MP^n$ is fully additive, namely,
\begin{align*}
Z (P_1 \cup \dots \cup P_m) = \sum_{j=1}^m \sum_{1 \leq i_1 < \dots <i_j \leq m}(-1)^{j-1} Z(P_{i_1} \cap \dots \cap P_{i_j}).
\end{align*}
Indeed, since $P \mapsto Z P (x)$ is a real valued valuation for any $x \in \R^n$, this is a direct corollary of the fact that real valued valuations on $\MP^n$ are fully additive \cite{Sch14}.

Set $T^d = [o,e_1,\dots,e_{d}]$ for $0 \leq d \leq n$.
If a valuation $Z$ is $\SLn$ contravariant, then $Z$ of every simplex containing the origin as one of their vertices is determined by $ZT^d$ for some $d$.
Also, $Z$ of every simplex contained in a hyperplane not going through the origin is determined by $Z[e_1,\dots,e_{d}]$ for some $d$.
If $P$ is a polytope containing the origin, we can use full additivity to calculate the valuation of $P$ by dissecting $P$ into simplices containing the origin as one of their vertices.
For $o \notin P$, we can dissect $[P,o]$ into $P$ and polytopes $[F_i,o]$, where $F_i$ are facets of $P$ visible from the origin.
Since $Z$ of $[P,o]$, $[F_i,o]$ and their intersections are determined by simplices, one can get the following uniqueness of valuations.
Details can be seen (for example) in \cite{LR2017sl}.

\begin{lem}\label{lemuq}
Let $Z$ and $Z'$ be $\SLn$ contravariant function valued valuations on $\MPon$. If $Z (sT^d) = Z' (sT^d)$ for every $s>0$ and $0 \leq d \leq n$, then $Z P = Z' P$ for every $P \in \MPon$.
\end{lem}

\begin{lem}\label{lemuq2}
Let $Z$ and $Z'$ be $\SLn$ contravariant function valued valuations on $\MP^n$. If $Z (sT^d) = Z' (sT^d)$ and $Z(s[e_1,\dots,e_d]) = Z'(s[e_1,\dots,e_d])$ for every $s>0$ and $0 \leq d \leq n$, then $Z P = Z' P$ for every $P \in \MP^n$.
\end{lem}

\section{Proof of the main results}\label{sec:proofthm}
In this section, we will always assume that $n \geq 3$.
\begin{lem}\label{lem3.1}
If $Z : \mathcal{P}_o ^n \to \mathcal{C}(\mathbb{R}^n)$ is $\SLn$ contravariant, then
$$Z P (x)= Z P (o),~x \in \R^n,$$
for every $P \in \mathcal{P}_o ^n$ satisfying $\dim P \leq n-2$, and
$$Z P (x) = Z P (x_n e_n),~x \in \R^n$$
for every $P \in \mathcal{P}_o ^n$ satisfying that $\dim P = n-1$ and $P \subset e_n^{\bot}$.
\end{lem}
\begin{proof}
Let $P \in \mathcal{P}_o ^n$ and $\dim P =d <n$. We can assume that the linear hull of $P$ is $ \lin \{ e_1, \dots, e_d \}$, the linear hull of $ \{ e_1, \dots, e_d \}$. Let $\phi := \left[ {\begin{array}{*{20}{c}}
I&A \\
0&B
\end{array}} \right] \in \SLn$, where $I \in \mathbb{R}^{d \times d}$ is the identity matrix, $A \in \mathbb{R}^{d \times (n-d)}$ is an arbitrary matrix, $B \in \SL{n-d}$, $0 \in \mathbb{R}^{(n-d)\times d}$ is the zero matrix. Also, let $x = \left( {\begin{array}{*{20}{c}}
{x'}\\
{x''}
\end{array}} \right) \in \mathbb{R}^{d \times (n-d)}$ and $x'' \neq 0$. Thus $\phi P = P$. By the $\SLn$ contravariance of $Z$, we have
\begin{align*}
Z P (x) = Z (\phi P) (x) = Z P (\phi ^{-1} x) = Z P \left( {\begin{array}{*{20}{c}}
{x' - AB^{-1}x''}\\
{B^{-1}x''}
\end{array}} \right).
\end{align*}

For $d \leq n-2$, we can choose a suitable matrix $B$ such that $B^{-1}x''$ is any nonzero vector on $\lin \{ e_{d+1},\dots,e_n \}$. After fixing $B$ we can also choose a suitable matrix $A$ such that $x' - AB^{-1}x''$ is any vector in $\lin \{ e_1, \dots, e_d \}$. So $Z P (x)$ is a constant function on a dense set of $\mathbb{R}^n$. By the continuity of $Z P$, we get $Z P (x) = Z P (o)$.

For $d = n-1$, we have $B=1$ and $x''=x_ne_n$. If $x_n \neq 0$, we can choose a suitable $A$ such that $x' - AB^{-1}x''=0$. Hence $$Z P (x) = Z P (x_n e_n)$$
for $x_n \neq 0$.
Now for $x_n=0$, the continuity of $Z P$ shows that
$$Z P \left( {\begin{array}{*{20}{c}}{x'}\\{0}\end{array}} \right) = \lim_{x_n \to 0}Z P \left( {\begin{array}{*{20}{c}}{x'}\\{x_n}\end{array}} \right) = \lim_{x_n \to 0}Z P (x_n e_n) = Z P(0 ~e_n).$$
\end{proof}

\begin{lem}\label{lem3.2}
If $Z : \mathcal{P}_o ^n \to \mathcal{C}(\mathbb{R}^n)$ is an $\SLn$ contravariant valuation satisfying $Z \{o\} (o) =0$ and $Z [o,e_1] (o)=0$,
then there exists a constant $c_{n-1} \in \R$ such that
\begin{align}\label{ctlowd}
Z P (x) = c_{n-1} V_1(P,[-x,x]),  ~~x \in \R^n
\end{align}
for every $P \in \mathcal{P}_o ^n$ satisfying $\dim P \leq n-1$.
\end{lem}
\begin{proof}
For $0 < \lambda <1$, let $H_\lambda = \{x \in\mathbb{R}^n : x \cdot ((1-\lambda) e_1- \lambda e_2) = 0\}$, $H_\lambda^- := \{ x \in\mathbb{R}^n : x \cdot ((1-\lambda) e_1- \lambda e_2) \leq 0 \}$ and $H_\lambda^+ := \{ x \in\mathbb{R}^n : x \cdot ((1-\lambda) e_1- \lambda e_2) \geq 0 \}$.
Since $Z$ is a valuation,
\begin{align}\label{val}
Z (sT^{d}) (x) + Z (sT^{d} \cap H_\lambda) (x) = Z (sT^{d} \cap H_\lambda ^-) (x) + Z (sT^{d} \cap H_\lambda ^-) (x), ~~x \in \R^n
\end{align}
for $2 \leq d \leq n$, $s > 0$.
Let $\widehat{T}^{d-1} = [o,e_1,e_3,\dots,e_d]$ and $\phi_1,\phi_2 \in \SLn$ such that
\begin{align*}
\phi _1 e_1 = &\lambda e_1 + (1-\lambda) e_2,~\phi _1 e_2 = e_2,~\phi _1 e_n = \frac{1}{\lambda} e_n,\\
&\phi _1 e_i = e_i,~~~\text{for}~3 \leq i \leq n-1
\end{align*}
and
\begin{align*}
\phi _2 e_1 = e_1,&~\phi _2 e_2 = \lambda e_1 + (1-\lambda) e_2,~\phi _2 e_n = \frac{1}{1-\lambda} e_n,\\
&\phi _2 e_i = e_i,~~~\text{for}~3 \leq i \leq n-1.
\end{align*}
For $2 \leq d \leq n-1$,
we have $T^{d}\cap H_\lambda ^- = \phi _1 T^{d}$, $T^{d}\cap H_\lambda ^+ = \phi _2 T^{d}$ and $T^d \cap H_\lambda =\phi _1 \widehat{T}^{d-1}$.
Also, since $Z$ is $\SLn$ contravariant, \eqref{val} implies that
\begin{align}\label{a7-1}
Z (sT^{d}) (te_n) + Z (s\widehat{T}^{d-1}) (\lambda te_n) = Z (sT^{d}) (\lambda te_n) + Z (sT^{d}) ((1-\lambda)te_n)
\end{align}
for $t \in \mathbb{R}$.
With $t=0$ in (\ref{a7-1}), we have $Z (sT^{d}) (o) = Z (s\widehat{T}^{d-1}) (o)$ for $d \leq n-1$. Using the $\SLn$ contravariance of $Z$ again, we have
$$Z (sT^{d}) (o) = Z (sT^{d-1}) (o)= \dots = Z (s[o,e_1]) (o) = Z [o,e_1] (o).$$
Combined with Lemma \ref{lem3.1} and the assumption $Z [o,e_1](o) =0$, we have
\begin{align}\label{1}
Z (sT^{d}) \equiv 0
\end{align}
for $s >0$ and $d \leq n-2$.

For $d = n-1$, the relations (\ref{a7-1}), \eqref{1} and the $\SLn$ contravariance of $Z$ show that
\begin{align}\label{a7-2}
Z T^{n-1} (te_n) = Z T^{n-1} (\lambda te_n) + Z T^{n-1} ((1-\lambda)te_n)
\end{align}
for $t \in \R$.
Let $f(t) : =Z T^{n-1} (te_n)$. For arbitrary $t_1,t_2 > 0 $, setting $t= t_1+t_2$, $\lambda =\frac{t_1}{t_1+t_2}$ in (\ref{a7-2}), we get that $f$ satisfies the Cauchy functional equation
\begin{align*}
f(t_1+t_2) = f(t_1)+f(t_2)
\end{align*}
for every $t_1,t_2 >0$. Since $f$ is continuous, there exists a constant $c_{n-1} \in \R$ such that
\begin{align*}
Z T^{n-1} (te_n) = f(t) = c_{n-1} t
\end{align*}
for $t \geq 0$. Also, since $Z$ is $\SLn$ contravariant,
$Z T^{n-1} (te_n) = Z T^{n-1} (-te_n).$
Hence $Z T^{n-1} (te_n) = c_{n-1} t$ holds for all $t \in \R$. The $\SLn$ contravariance of $Z$ now shows that
$$Z (sT^{n-1}) (te_n) = Z T^{n-1} (s^{n-1} te_n) = c_{n-1} s^{n-1}t$$
Combined with Lemma \ref{lem3.1} and (\ref{pi1}), we have
\begin{align}\label{2}
Z (sT^{n-1}) (x) = c_{n-1} \frac{n!}{2} V_1(sT^{n-1},[-x,x]).
\end{align}
We replace $c_{n-1} \frac{n!}{2}$ by $c_{n-1}$.
Now (\ref{1}) and (\ref{2}) imply that (\ref{ctlowd}) holds for $T^d$ for $0 \leq d \leq n-1$.
Since $Z$ is $\SLn$ contravariant, we can argue as in Lemma \ref{lemuq} to show that (\ref{ctlowd}) holds for $o \in P \subset \R^{n-1}$.
Every lower dimensional polytope containing the origin can be rotated to be contained in $P \subset \R^{n-1}$.
Hence we get the desired results.
\end{proof}

Next we deal with simple valuations.
\begin{lem}\label{lem4.3}
If $Z : \mathcal{P}_o ^n \to \mathcal{C}(\mathbb{R}^n)$ is a simple and $\SLn$ contravariant valuation and the function $r \mapsto Z(rT^n)(rte_n), ~r>0$ is measurable for any $t \in \R$,
then there is a continuous function $\zeta: \R \to \R$ such that
$$Z (sT^n) (te_n) = \frac{s^n}{n!} \zeta\left(\frac{t}{s}\right)
= \int_{\Sn \setminus \{h_{sT^n}=0\}} \zeta \left(\frac{te_n \cdot u}{h_{sT^n}(u)}\right)dV_{sT^n}(u)$$
for $s>0$ and $t \in \R$.
\end{lem}
\begin{proof}
The second equation is trivial. We only need to verify the first equation.

Let $\phi_3, \phi_4 \in \SLn$ such that
\begin{align*}
\phi_3 e_1 = &\lambda^{-1/n} (\lambda e_1 + (1-\lambda) e_2),~\phi_3 e_2 = \lambda^{-1/n} e_2, \\
&\phi_3 e_i = \lambda^{-1/n} e_i,~\text{for}~3 \leq i \leq n,
\end{align*}
and
\begin{align*}
\phi_4 e_1 = &(1-\lambda)^{-1/n} e_1,~\phi_4 e_2 = (1-\lambda)^{-1/n} (\lambda e_1 + (1-\lambda) e_2), \\
&\phi_4 e_i = (1-\lambda)^{-1/n} e_i,~\text{for}~3 \leq i \leq n.
\end{align*}
We use the same notation as in Lemma \ref{lem3.2}.
Note that $sT^{n}\cap H_\lambda ^- = \phi_3 \lambda^{1/n} sT^{n}$, $sT^{n}\cap H_\lambda ^+ = \phi_4 (1-\lambda)^{1/n}sT^{n}$ and $sT^n \cap H_\lambda =\phi _3 \lambda^{1/n}s\widehat{T}^{n-1}$. The valuation property (\ref{val}) for $d=n$ together with the $\SLn$ contravariance and simplicity of $Z$ shows that
\begin{align}\label{a7}
Z (sT^{n}) (x) = Z (\lambda^{1/n} sT^{n}) (\phi_3 ^{-1} x) + Z ((1-\lambda)^{1/n} sT^{n}) (\phi_4^{-1}x).
\end{align}
%
For $t' \in \mathbb{R}$, choosing $x = t' e_n$ in
\eqref{a7}, we have
\begin{align}\label{a7-3}
Z (sT^{n}) ( t'e_n)
=Z (\lambda^{1/n} sT^{n}) (\lambda^{1/n} t' e_n) + Z ((1-\lambda)^{1/n} sT^{n}) ((1-\lambda)^{1/n}t' e_n)
\end{align}
for any $0 < \lambda <1$ and $s>0$.
Let
\begin{align}\label{Cau1}
f(t;r) = Z (r^{1/n} T^{n}) (r^{1/n} te_n)
\end{align} for $r>0$.
For arbitrary $r_1,r_2 > 0 $, $t \in \R$, setting $s= (r_1+r_2)^{1/n}$, $t'= (r_1+r_2)^{1/n}t$, $\lambda =\frac{r_1}{r_1+r_2}$ in (\ref{a7-3}), we get that $f$ satisfies the Cauchy functional equation
\begin{align*}
f(t;r_1+r_2)
= f(t;r_1) + f(t;r_2).
\end{align*}
Since the function $r \mapsto Z(rT^n)(rte_n), ~r>0$ is measurable for any $t \in \R$, so is $f(t;\cdot)$.
Therefore there exists a constant $c(t)$ such that
\begin{align*}
Z (r^{1/n} T^{n}) (r^{1/n} te_n) = f(t;r) = c(t) r
\end{align*}
for every $r > 0$ and $t \in \R$. Hence
$$Z (sT^{n}) (te_n) = c(t/s)s^n.$$
Since $t \mapsto Z (T^{n}) (te_n)$ is continuous, $c(t)$ is also continuous.
Now setting $\zeta(t) = n!c(t)$ completes the proof.
\end{proof}

\begin{lem}\label{lem4.4}
If $Z : \mathcal{P}_o ^n \to \mathcal{C}(\mathbb{R}^n)$ is a simple and $\SLn$ contravariant valuation and $Z(sT^n)(te_n) =0$ for any $s>0$, $t \in \R$, then
\begin{align}\label{valct2}
Z(sT^n)(x) = 0
\end{align}
for any $x \in \R^n$.
\end{lem}
\begin{proof}
We will use induction on the number $m$ of coordinates of $x$ not equal to zero.
Since $Z$ is $\SLn$ contravariant, we assume that the first $m$ coordinates $x_1,\dots,x_m$ are not zero.
The assumption $Z(sT^n)(te_n) =0$ and the $\SLn$ contravariance of $Z$ show that \eqref{valct2} holds for $m=1$. Now assume that (\ref{valct2}) holds for $m-1$. Let $\dot{x}=x_3 e_3 + \dots + x_me_m$.

If $x_1, x_2$ have the same sign, then taking $x=x_1 e_1 +x_2 e_2+\dot{x}$ and $\lambda = \frac{x_1}{x_1+x_2}$ in (\ref{a7}), we obtain
\begin{align}\label{ct1}
&Z\left(sT^{n}\right) \left(x_1 e_1 +x_2 e_2 + \dot{x}\right) \nonumber \\
&= Z\left(\lambda^{1/n}sT^{n}\right) \left(\lambda^{1/n} \left(\left(x_1+x_2\right)e_1 + \dot{x}\right)\right) \nonumber \\
& \qquad \qquad + Z\left(\left(1-\lambda\right)^{1/n}sT^{n}\right) \left(\left(1-\lambda\right)^{1/n}\left(\left(x_1+x_2\right)e_2 + \dot{x}\right)\right).
\end{align}

If $x_1 > -x_2 >0$ or $-x_1 > x_2 >0$, then taking $x= \lambda^{-1/n}((x_1 +x_2) e_1+\dot{x})$, $\lambda = \frac{x_1+x_2}{x_1}$ and $s=\lambda ^{-1/n}s$ in (\ref{a7}), we obtain
\begin{align}\label{ct2}
&Z\left(\lambda^{-1/n}sT^{n}\right) \left(\lambda^{-1/n}\left(\left(x_1 +x_2\right) e_1 + \dot{x}\right)\right) \nonumber\\
&=Z\left(sT^{n}\right) \left(\left(x_1 e_1+x_2 e_2 + \dot{x}\right)\right) \nonumber \\
&\qquad \qquad + Z\left(\lambda^{-1/n}\left(1-\lambda\right)^{1/n}sT^{n}\right) \left(\lambda^{-1/n}\left(1-\lambda\right)^{1/n}\left(\left(x_1+x_2\right)e_1 + \dot{x}\right) \right).
\end{align}

If $x_2 > -x_1 >0$ or $-x_2 > x_1 >0$, then taking $x= (1-\lambda) ^{-1/n}((x_1 +x_2) e_2+\dot{x})$, $\lambda = -\frac{x_1}{x_2}$ and $s=(1-\lambda) ^{-1/n}s$ in \eqref{a7}, we obtain
\begin{align}\label{ct3}
&Z\left(\left(1-\lambda\right) ^{-1/n}sT^{n}\right) \left(\left(1-\lambda\right) ^{-1/n}\left(\left(x_1 +x_2\right) e_2 +\dot{x}\right)\right) \nonumber \\
&=Z\left(\left(1-\lambda\right) ^{-1/n} \lambda^{1/n} sT^{n}\right) \left(\left(1-\lambda\right) ^{-1/n} \lambda^{1/n} \left(\left(x_1 +x_2\right) e_2+\dot{x}\right)\right) \nonumber \\
&\qquad \qquad + Z\left(sT^{n}\right) \left(x_1 e_1+x_2 e_2 + \dot{x}\right).
\end{align}

Now that \eqref{valct2} holds for $m$ follows directly from the induction assumption together with (\ref{ct1}), (\ref{ct2}), (\ref{ct3}) and the continuity of $Z(sT^{n})$.
\end{proof}

Before proving Theorem \ref{thm:cont}, we first show a slightly stronger result. This result will be used for Corollaries \ref{qhom} and \ref{co:mixv}.
\begin{thm1.2'}
Let $Z : \mathcal{P}_o ^n \to \mathcal{C}(\mathbb{R}^n)$ be an $\SLn$ contravariant valuation. If the function $r \mapsto Z(rT^n)(rte_n), ~r>0$ is measurable for any $t \in \R$, then
there are constants $c_0,c_0',c_{n-1} \in \R$ and a continuous function $\zeta : \R \to \R$ such that
\begin{align*}
Z P (x)&=\int_{\Sn \setminus \{h_P=0\}} \zeta \left(\frac{x \cdot u}{h_{P}(u)}\right)d V_P(u) + c_{n-1} V_1(P,[-x,x])  \notag \\
&\qquad + c_0V_0(P) + c_0' (-1)^{\dim P} \mathbbm{1}_{\relint P}(o)
\end{align*}
for every $P \in \mathcal {P}_o^n$ and $x \in \R^n$. Moreover, $c_0,c_0',c_{n-1}$ and $\zeta$ are uniquely determined by $Z$.
\end{thm1.2'}
\begin{proof}
Let $Z : \mathcal{P}_o ^n \to \mathcal{C}(\mathbb{R}^n)$ be an $\SLn$ contravariant valuation.
Set $c_0 := Z [o,e_1] (o)$ and $c_0' = Z \{o\} (o)-c_0$.
The new valuation $Z'P = Z P - c_0V_0(P) - c_0'(-1)^{\dim P} \mathbbm{1}_{\relint P}(o)$ is an $\SLn$ contravariant valuation satisfying $Z' \{o\} (o) =0$ and $Z' [o,e_1] (o)=0$.
By Lemma \ref{lem3.2}, we have
\begin{align*}
Z' P (x)- c_0V_0(P) - c_0'(-1)^{\dim P} \mathbbm{1}_{\relint P}(o) = c_{n-1}V_1(P,[-x,x])
\end{align*}
for every $x\in \R^n$ and $P \in \mathcal{P}_o ^n$ satisfying $\dim P \leq n-1$. Now let $Z''(P) (x)= Z P (x)- c_0V_0(P) - c_0'(-1)^{\dim P} \mathbbm{1}_{\relint P}(o) - c_{n-1}V_1(P,[-x,x])$ for every $P \in \mathcal{P}_o ^n$ and $x \in \R^n$.
Then $Z''$ is a simple and $\SLn$ contravariant valuation. Also the function $r \mapsto Z''(rT^n)(rte_n), ~r>0$ is measurable for any $t \in \R$. Similarly Lemma \ref{lem4.3} and Lemma \ref{lem4.4} together show that there is a continuous function $\zeta : \R \to \R$ such that
\begin{align*}
Z''(sT^n)-Z_\zeta(sT^n) =0.
\end{align*}
Here $Z_\zeta P (x) = \int_{\Sn \setminus \{h_P=0\}} \zeta \left(\frac{x \cdot u}{h_{P}(u)}\right)d V_P(u)$ was studied in Theorem \ref{thm:val}.
Lemma \ref{lemuq} now shows that $Z'' P-Z_\zeta P=0$ for every $P \in \MPon$. Hence
$$Z P (x) = Z_\zeta P (x)+ c_{n-1}V_1(P,[-x,x])+ c_0V_0(P) + c_0'(-1)^{\dim P} \mathbbm{1}_{\relint P}(o)$$
for every $P \in \MP_o^n$ and $x\in \R^n$. Clearly $c_0,c_0'$ are uniquely determined by $Z[o,e_1]$ and $Z\{o\}$. Hence $c_{n-1}$ is uniquely determined by $Z$ on $(n-1)$-dimensional polytopes. Finally $\zeta$ is uniquely determined by $Z T^n$.
\end{proof}

\begin{proof}[Proof of Theorem \ref{thm:cont}]
We only need to show that the measurability of $Z$ implies the measurability of the function $r \mapsto Z(rT^n)(rte_n), ~r>0$ for any $t \in \R$. For fixed $t \in \R$, define functions
\begin{align*}
F_1 : ~&\R \to \R^n \times \MPon  &F_2:~ &\R^n \times \MPon \to  \R^n \times \mathcal{C}(\R^n)   \\
&r~  \mapsto (rte_n, r T^n)   &     &~(x,P) ~~\mapsto (x,ZP)
\end{align*}
and
\begin{align*}
F_3:~ &\R^n \times \mathcal{C}(\R^n) \to \R \\
  &~~~(x,g) ~~~~~\mapsto g(x).
\end{align*}
Clearly $F_1$ is continuous and $F_2$ is measurable follows from the assumptions. The evaluation map $F_3$ is continuous; see \cite[Theorem 46.10]{Mun2000topo}. Hence $Z(rT^n)(rte_n) = F_3 \circ F_2 \circ F_1 (r)$ is measurable.
\end{proof}

\begin{proof}[Proof or Theorem \ref{thm:MKo}]
Let $Z$ be a valuation satisfying all conditions.
Theorem \ref{thm:cont} shows that $Z$ has the representation \eqref{val22} on $\MPon$.
Since $Z_\zeta $ is simple and $Z, V_1(\cdot,[-x,x]), V_0$ are continuous valuations on lower dimensional polytopes, $c_0'=0$.

Now we need to show that the continuity of $Z_\zeta$ implies $\lim_{|t|\to \infty}\zeta(t)/t =0$. Let $P=\sum_{i=1}^{n-1}[-e_i,e_i] + [o,e_n]$ and $P_t=\sum_{i=1}^{n-1}[-e_i,e_i] + [-\frac{1}{t} e_n,e_n]$ for $t >0$. Clearly $P_t \to P$ when $t \to \infty$. Hence, we have
\begin{align*}
\lim_{t \to \infty} \int_{\Sn} \zeta \left(\frac{x \cdot u}{h_{P_t}(u)}\right)d V_{P_t}(u) = \int_{\Sn \setminus \{-e_n\}} \zeta \left(\frac{x \cdot u}{h_{P}(u)}\right)d V_P(u)
\end{align*}
for any $x \in \R^n$.
Thus
\begin{align*}
\lim_{t \to \infty}\zeta \left(-x_n t\right)\frac{1}{nt}2^{n-1} = 0
\end{align*}
for any $x_n \in \R$. By taking $x_n = \pm 1$, we obtain
\begin{align*}
\lim_{|t| \to \infty}\zeta \left(t\right)/t = 0.
\end{align*}
The desired result now follows from the fact that $\MP_o^n$ is a dense subset of $\MK_o^n$. The uniqueness of $c_0,c_{n-1},\zeta$ follows as in Theorem \ref{thm:cont}.
\end{proof}

\begin{proof}[Proof of Theorems \ref{thm:MK} and \ref{thm:MP}]
Let $Z$ be a valuation satisfying all conditions. First for Theorem \ref{thm:MP}, by Lemma \ref{lemuq2}, we only need to show that $Z$ has the corresponding representation at $sT^d$ and $s[e_1,\dots,e_d]$.
Applying Theorem \ref{thm:cont} to $\MPon$, there are constants $a_0, a_0',a_{n-1}$ and a continuous $\xi: \R \to \R$ such that
\begin{align*}
&Z(sT^d)(x)  \\
&= Z_\xi (sT^d)(x) + a_{n-1}V_1(sT^d[-x,x])+ a_0V_0(sT^d) + a_0'(-1)^{d} \mathbbm{1}_{\relint (sT^d)}(o)
\end{align*}
for every $x \in \R^n$.

Let $\MTon $ be the set of simplices in $\R^n$ with one vertex at the origin.
For any $T \in \MTon \setminus \{o\}$, we write $T'$ for its facet opposite to the origin.
We define the new map $\widetilde{Z}: \MTon \setminus \{o\} \to \mathcal{C}(\R^n)$ by $\widetilde{Z} (T) = Z (T')$ for any $T \in \MTon \setminus \{o\}$.
It is not hard to show that $\widetilde{Z}$ is a measurable $\SLn$ contravariant valuation on $\MTon \setminus \{o\}$.
In the proof of Theorem \ref{thm:cont}, we actually proved that Theorem \ref{thm:cont} also holds on $\MTon \setminus \{o\}$. Hence there are constants $b_0, b_0',b_{n-1}$ and a continuous $\widetilde{\xi}: \R \to \R$ such that
\begin{align*}
Z (s [e_1,\dots,e_d])(x) &= \widetilde{Z}(sT^d)(x)  \\
&= Z_{\widetilde{\xi}} (sT^d)(x) + b_{n-1}V_1(sT^d[-x,x])+ b_0V_0(sT^d)
\end{align*}
for every $x \in \R^n$ (The term $(-1)^{d} \mathbbm{1}_{\relint (sT^d)}(o)$ does not appear since it only depends on the valuation at $o$).

Now we choose new constants $c_0,c_0',\widetilde{c}_0,c_{n-1},\widetilde{c}_{n-1}$ and continuous functions $\zeta, \widetilde{\zeta}$ such that
\begin{align*}
\zeta(t)= \xi(t) - \widetilde{\xi}(t) + 2(a_{n-1}-b_{n-1})|t|,~~ \widetilde{\zeta}(t)=\widetilde{\xi}(t)-2(a_{n-1}-b_{n-1})|t|
\end{align*}
for $t \in \R$ and
\begin{align*}
c&_{n-1} = a_{n-1}- b_{n-1}, ~~  \widetilde{c}_{n-1}=b_{n-1}, \\
&c_0=b_0,~~c_0'=a_0',~~\widetilde{c}_0=a_0-b_0.
\end{align*}
Hence \eqref{val24} holds for $sT^d$ and $[se_1,\dots,e_d]$ for $0 \leq d \leq n$, which completes the proof of Theorem \ref{thm:MP}.

Now using the continuity of $Z$ on $1$-dimensional polytopes, we have $c_0'=\widetilde{c}_0=0$. Also, since $V_0(K)$, $V_1(K,[-x,x])$ and $V_1([K,o],[-x,x])$ are continuous valuations, similarly to the proof of Theorem \ref{thm:MKo}, we only need to show the fact that
\begin{align*}
\int_{\Sn \setminus \{h_P =0\}} \zeta \left(\frac{x \cdot u}{h_{P}(u)}\right)d V_P(u)  + \int_{\Sn \setminus \{h_{[P,o]} =0\}} \widetilde{\zeta} \left(\frac{x \cdot u}{h_{[P,o]}(u)}\right)d V_{[P,o]}(u)
\end{align*}
is a continuous valuation implies $\lim_{|t| \to \infty}\zeta \left(t\right)/t = 0$ and $\lim_{|t| \to \infty}\widetilde{\zeta} \left(t\right)/t = 0$. Let $P \in \MPon$. In the proof of Theorem \ref{thm:cont}, we have already shown that
\begin{align*}
\lim_{|t| \to \infty} \frac{\zeta \left(t\right)+ \widetilde{\zeta}\left(t\right)}{t}= 0.
\end{align*}
Now let $P_t = \sum_{i=1}^{n-1}[-e_i,e_i] + \frac{1}{t} e_n$ and $P= \sum_{i=1}^{n-1}[-e_i,e_i]$. Hence $P_t \to P$ when $t \to \infty$. Similarly to the proof of Theorem \ref{thm:cont}, we get
\begin{align*}
\lim_{|t| \to \infty} \widetilde{\zeta} \left(t\right)/t= 0,
\end{align*}
which completes the proof of Theorem \ref{thm:MK}.
\end{proof}

\begin{section}{Proof of the corollaries}\label{sec:proofco}
Let $Z$ be a valuation satisfying all conditions. We need to prove that $Z$ has the corresponding representation in all corollaries. In the following, we always let $P \in \MPon$.
\begin{proof}[Proof of Corollaries \ref{p>=1} and \ref{p<1}]
Let $p \geq 0$. Since $Z P(\lambda x) = \lambda^p Z P(x)$ for any $\lambda >0$, by Theorem \ref{thm:cont}, we have
\begin{align*}
\lambda^p Z P (x)&=\int_{\Sn \setminus \{h_P=0\}} \zeta \left(\frac{\lambda x \cdot u}{h_{P}(u)}\right)d V_P(u) + c_{n-1} \lambda V_1(P,[-x,x])  \\
&\qquad + c_0V_0(P) + c_0' (-1)^{\dim P} \mathbbm{1}_{\relint P}(o)
\end{align*}
for any $\lambda>0$ and $x \in \R^n$. Comparing coefficients, we have
\begin{align}\label{phom4}
\lambda^p Z P (x)&=\int_{\Sn \setminus \{h_P=0\}} \zeta \left(\frac{\lambda x \cdot u}{h_{P}(u)}\right)d V_P(u)
\end{align}
for $p>0$ and $p \neq 1$,
\begin{align}\label{phom6}
\lambda Z P (x)&=\int_{\Sn \setminus \{h_P=0\}} \zeta \left(\frac{\lambda x \cdot u}{h_{P}(u)}\right)d V_P(u)+ c_{n-1} \lambda V_1(P,[-x,x])
\end{align}
for $p=1$ and
\begin{align}\label{phom5}
Z P (x)&=\int_{\Sn \setminus \{h_P=0\}} \zeta \left(\frac{\lambda x \cdot u}{h_{P}(u)}\right)d V_P(u) + c_0V_0(P) + c_0'(-1)^{\dim P} \mathbbm{1}_{\relint P}(o)
\end{align}
for $p=0$.
Now let $P=T^n$, $x=\pm e_n$, \eqref{phom4} induces that
\begin{align*}
\lambda^p Z T^n (\pm e_n) =\frac{\zeta \left(\pm \lambda \right)}{n!}.
\end{align*}
Let $\hat{c}_{n-p}^+ = n! Z T^n (e_n)$ and $\hat{c}_{n-p}^- = n! Z T^n (-e_n)$. Thus $\zeta(t) = \hat{c}_{n-p}^+ (t)_+^p + \hat{c}_{n-p}^- (t)_-^p$, $t \in \R^n$ for $p>0$ and $p \neq 1$. Similarly, \eqref{phom6} implies that
$\zeta(t) =\hat{c}_{n-p}^+ (t)_+ + \hat{c}_{n-p}^- (t)_-$, $t \in \R$
for suitable $\hat{c}_{n-p}^+, \hat{c}_{n-p}^- \in \R$,
and \eqref{phom5} implies that
\begin{align*}
\zeta(t) =c_n^+,  ~~ \zeta(-t) =c_n^-
\end{align*}
for suitable constants $c_n^+, c_n^-$ when $t>0$. Since $\zeta$ is a continuous function, $c_n^+=c_n^-$. Hence $\zeta \equiv c_n$ for a suitable constant $c_n$.
Now back to \eqref{phom4}-\eqref{phom5}, we obtain the desired result.
\end{proof}

\begin{proof}[Proof of Corollary \ref{co:Orl}]
First note that a convex function from $\R^n$ to $\R$ is continuous.
Since $Z$ vanishes at $[o,e_1]$, $\{o\}$, $T^{n-1}$, step by step, we get that $c_0,c_0',c_{n-1}=0$ in Theorem \ref{thm:cont}. Now
\begin{align*}
Z T^n  (t e_n) = \frac{1}{n!} \zeta(t)
\end{align*}
for any $t \in \R$. Since $Z T^n$ is a convex function, $\zeta$ is also convex.
\end{proof}

\begin{proof}[Proof of Corollary \ref{qhom}]
Let $p \in \R$.
Since $Z T^n(\cdot)$ is a continuous function and we further assume that $Z (\lambda T^n) = \lambda^{n-p} Z T^n$ for $\lambda>0$, it follows that the function $r \mapsto Z(rT^n)(rte_n)$ is continuous on $(0, \infty)$.
By Theorem 1.2$'$, we get
\begin{align*}
\lambda^{n-p} Z P (x)&=\int_{\Sn \setminus \{h_P=0\}} \zeta \left(\frac{x \cdot u}{\lambda h_{P}(u)}\right)\lambda^n d V_P(u) + c_{n-1} \lambda^{n-1} V_1(P,[-x,x])  \\
&\qquad + c_0V_0(P) + c_0' (-1)^{\dim P} \mathbbm{1}_{\relint P}(o)
\end{align*}
for any $\lambda>0$ and $x \in \R^n$. Now using similar arguments as in the proof of Corollaries \ref{p>=1} and \ref{p<1}, we get the desired result.
\end{proof}

\begin{proof}[Proof of Corollary \ref{trans}]
First applying the translation invariance on lower dimensional convex bodies in Theorem \ref{thm:cont}, we obtain that $c_0' =0$. We only need to show that the $\zeta$ in Theorem \ref{thm:cont} is now a constant function.
Let $P = \sum_{i=1}^n[-e_i,e_i]$ and $-1 \leq t \leq 1$. Since $Z$ is translation invariant, we have
$Z (P+te_n) (re_n) = Z (P)(r e_n)$ for any $r\in\R$.
Together with Theorem \ref{thm:cont},
\begin{align}\label{Cau2}
\zeta \left(\frac{r}{1+t}\right)(1+t)2^{n-1} + \zeta\left(\frac{-r}{1-t}\right)(1-t)2^{n-1} = \zeta \left(r\right)2^{n-1} + \zeta\left(-r\right)2^{n-1}
\end{align}
for $-1 < t <1$ and
\begin{align}\label{Cau3}
\zeta \left(\frac{r}{2}\right)2^{n}  = \zeta \left(r\right)2^{n-1} + \zeta\left(-r\right)2^{n-1}.
\end{align}
Let $f(t)= \zeta\left(\frac{1}{t}\right)t$ for $t \neq 0$.
The relation \eqref{Cau3} implies that
\begin{align*}
f(2t) = f(t) - f(-t)
\end{align*}
for any $t \neq 0$. Now changing $t$ to $-t$, we get that
\begin{align}\label{odd}
f(-t) = -f(t)
\end{align}
Let $t_1=-\frac{1-t}{r}$ and $t_2= \frac{2}{r}$. Hence $t_1+t_2=\frac{1+t}{r}$. Back to \eqref{Cau2} and \eqref{Cau3}, we obtain that
\begin{align}\label{Cau4}
f(t_1) + f(t_2) = f(t_1+t_2)
\end{align}
for any $t_2 \neq 0$ and $t_1 \in (0,-t_2)$. Set $f(0)=0$. Together with \eqref{odd}, the equation \eqref{Cau4} holds for any $t_2 \in \R$ and $t_1 \in [0,-t_2]$. Now let $t_1 \in (0,t_2]$. We have
\begin{align*}
f(t_1+t_2) + f(-t_1) = f(t_2).
\end{align*}
Also by \eqref{odd},
\begin{align*}
 -f(-t_1)+f(t_2) = f(t_1) + f(t_2).
\end{align*}
Thus \eqref{Cau4} holds for any $t_2 \in \R$ and $|t_1| \leq |t_2|$.
Now changing the order of $t_1,t_2$, we find that \eqref{Cau4} holds for any $t_1,t_2 \in \R$. Since $f(t)$ is continuous on $t \neq 0$, there exists a constant $c_n \in \R$ such that $f(t) = c_n t$. Recall that $f(t)= \zeta\left(\frac{1}{t}\right)t$ for $t \neq 0$. Combined with continuity of $\zeta$, we finally get $\zeta(t) = c_n$.
\end{proof}

\begin{proof}[Proof of Corollary \ref{trans2}]
Applying the translation invariance on lower dimensional convex bodies in Theorem \ref{thm:MP}, we obtain that $\widetilde{c}_{n-1} = c_0'=\widetilde{c}_0=0$. Further applying translation invariance on $\MPon$ in the proof of Corollary \ref{trans}, we get
\begin{align*}
\zeta + \widetilde{\zeta} \equiv c_n
\end{align*}
for a suitable constant $c_n$. Now let $P_t = \sum_{i=1}^{n-1}[-e_i,e_i] + t e_n$ for $t \in \R$. Since $Z (P_t) = Z (P_0)$, we have $\widetilde{\zeta} = 0$, which completes the proof.
\end{proof}

\begin{proof}[Proof of Corollary \ref{co:mixv}]
Clearly the representation of $Z$ satisfies all the conditions.
Now let $Z : \MPon \times \MK_c^n \to \R$ be an $\SLn$ invariant map. Set $Z' P (x) := Z(P,[-x,x])$ for $P \in \MPon$ and $x \in \R^n$.
Since $Z$ is a valuation with respect to the first variable and continuous with respect to the second variable, $Z'$ is a $\mathcal{C}(\R^n)$ valued valuation.
For fixed $P \in \MPon$ and $x \in \R^n$, let $f(t) = Z(P,[-t^{1/p}x,t^{1/p}x])$ for $t \geq 0$. Since
\begin{align*}
Z(P,[-(t_1+t_2)^{1/p}  x,(t_1+t_2)^{1/p}x])
&= Z(P,[-t_1^{1/p} x,t_1^{1/p}x]+_p [-t_2^{1/p}x,t_2^{1/p}x])\\
&= Z(P,[-t_1^{1/p} x, t_1^{1/p}x]) + Z(P,[-t_2^{1/p} x, t_2^{1/p}x])
\end{align*}
for any $t_1,t_2 \geq 0$ and $x \in \R^n$, we have $f(t_1+t_2) = f(t_1)+f(t_2)$ for any $t_1,t_2 \geq 0$. Hence $Z'P(tx) = f(t^p) = t^pf(1) = t^pZ'P(x)$ for $t \geq 0$ and $x\in \R^n$.
Together with the measurability of $Z$ with respect to the first variable, we obtain that the function $r \mapsto Z'(rT^n)(rte_n) = r^p Z'(rT^n)(te_n)= r^p Z(rT^n,[-te_n,te_n])$ is measurable. Hence by Theorem 1.2$'$ (similar to the proof of Corollary \ref{p>=1}) and the symmetry of the function $Z'P$, there are constants $\widehat{c}_{n-p},c_{n-1}\in \R$ such that
\begin{align*}
&Z(P,[-x,x]) \\
&=\widehat{c}_{n-p} \int_{\Sn \setminus \{h_P=0\}} |x \cdot u|^p h_P^{1-p}(u)d S_P(u) + \delta_p^1 c_{n-1} V_1(P,[-x,x])
\end{align*}
for every $P \in \MPon$ and $x \in \R^n$. Now the continuity and $L_p$ additivity of $Z$ with respect to the second variable shows that
\begin{align*}
Z (P,L) =  \widehat{c}_{n-p} \widehat{V}_p(P,L) + c_{n-1}\delta_p^1 V_1(P,L)
\end{align*}
for general $L_p$ zonoids. Here $L \in \MK_c^n$ is a general $L_p$ zonoids if $h_L^p (x) = \int_{\Sn} |x \cdot u|^p d \mu(u)$ for a signed Borel measure $\mu$ on $\Sn$. Also since the set of general $L_p$ zonoids is a dense subset of $\MK_c^n$ for $p$ not even (by combining \cite{Rub1998inv} with \cite[Theorem 3.4.1]{Sch14}), we get the desired result.
\end{proof}
\end{section}

\section*{Acknowledgement}
\addcontentsline{toc}{section}{Acknowledgement}
The authors wish to thank referees for careful reading and many valuable suggestions.
The work of the author was supported in part by the Austrian Science Fund (FWF), Project number: I 3027, the European Research Council (ERC), Project number: 306445, and the National Natural Science Foundation of China Project number: 11671249.



\addcontentsline{toc}{section}{References}
\begin{thebibliography}{10}
\providecommand{\url}[1]{\texttt{#1}}
\providecommand{\urlprefix}{URL }
\expandafter\ifx\csname urlstyle\endcsname\relax
  \providecommand{\doi}[1]{doi:\discretionary{}{}{}#1}\else
  \providecommand{\doi}{doi:\discretionary{}{}{}\begingroup
  \urlstyle{rm}\Url}\fi

\bibitem{abardia2011p}
\textit{J.~Abardia} and \textit{A.~Bernig}, Projection bodies in complex vector
  spaces, Adv. Math. \textbf{227} (2011), no.~2, 830--846.

\bibitem{Ale99}
\textit{S.~Alesker}, Continuous rotation invariant valuations on convex sets,
  Ann. of Math. \textbf{149} (1999), no.~3, 977--1005.

\bibitem{Ale01}
\textit{S.~Alesker}, {Description of translation invariant valuations on convex
  sets with solution of P. McMullen's conjecture}, Geom. Funct. Anal.
  \textbf{11} (2001), no.~2, 244--272.

\bibitem{AS2017+}
\textit{S.~Alesker} and \textit{F.~E. Schuster}, Affine invariant bivaluations, preprint.

\bibitem{BF2011herm}
\textit{A.~Bernig} and \textit{J.~H. Fu}, Hermitian integral geometry, Ann. of
  Math.  (2011), 907--945.

\bibitem{Blas1937vor}
\textit{W.~Blaschke}, {Vorlesungen {\"u}ber Integralgeometrie}, 20, BG Teubner,
  1937.

\bibitem{Bor2013}
\textit{K.~J. B\"{o}r\"{o}czky}, {Stronger versions of the Orlicz-Petty
projection inequality}, J. Differential Geom. \textbf{95} (2013), no.~2,
  215--247.

\bibitem{BL2017Min}
\textit{K.~J. B\"{o}r\"{o}czky} and \textit{M.~Ludwig}, {Minkowski valuations
  on lattice polytopes}, J. Eur. Math. Soc. (2017), in press.

\bibitem{BLYZ12}
\textit{K.~J. B\"{o}r\"{o}czky}, \textit{E.~Lutwak}, \textit{D.~Yang} and
  \textit{G.~Zhang}, {The log-Brunn-Minkowski inequality}, Adv. Math.
  \textbf{231} (2012), 1974--1997.

\bibitem{BLYZ13}
\textit{K.~J. B\"{o}r\"{o}czky}, \textit{E.~Lutwak}, \textit{D.~Yang} and
  \textit{G.~Zhang}, The logarithmic minkowski problem, J. Amer. Math. Soc.
  \textbf{26} (2013), no.~3, 831--852.

\bibitem{BLYZ14}
\textit{K.~J. B\"{o}r\"{o}czky}, \textit{E.~Lutwak}, \textit{D.~Yang} and
  \textit{G.~Zhang}, Affine images of isotropic measures, J. Differential Geom.
  \textbf{99} (2015), no.~3, 407--442.

\bibitem{CLZ2017lpMA}
\textit{S.~Chen}, \textit{Q.-R. Li} and \textit{G.~Zhu}, {On the $L_p$
  Monge--Amp{\`e}re equation}, J. Differential Equations \textbf{263} (2017),
  no.~8, 4997--5011.

\bibitem{CLM2017BM}
\textit{A.~Colesanti}, \textit{G.~V. Livshyts} and \textit{A.~Marsiglietti},
  {On the stability of Brunn--Minkowski type inequalities}, J. Funct. Anal.
  \textbf{273} (2017), no.~3, 1120--1139.

\bibitem{CLM2017Min}
\textit{A.~Colesanti}, \textit{M.~Ludwig} and \textit{F.~Mussnig}, Minkowski
  valuations on convex functions, Calc. Var. Partial Differential Equations
  \textbf{56} (2017), no.~6, 56:162.

\bibitem{GHW14}
\textit{R.~J. Gardner}, \textit{D.~Hug} and \textit{W.~Weil}, {The
  Orlicz-Brunn-Minkowski theory: A general framework, additions, and
  inequalities}, J. Differential. Geom. \textbf{97} (2014), no.~3, 427--476.

\bibitem{Hab09}
\textit{C.~Haberl}, Star body valued valuations, Indiana Univ. Math. J.
  \textbf{58} (2009), no.~5, 2253--2276.

\bibitem{Hab12b}
\textit{C.~Haberl}, Minkowski valuations intertwining with the special linear
  group, J. Eur. Math. Soc. \textbf{14} (2012), no.~5, 1565--1597.

\bibitem{HL06}
\textit{C.~Haberl} and \textit{M.~Ludwig}, {A characterization of $L_p$
  intersection bodies}, Int. Math. Res. Not. \textbf{10548} (2006), 1--29.

\bibitem{HLYZ10}
\textit{C.~Haberl}, \textit{E.~Lutwak}, \textit{D.~Yang} and \textit{G.~Zhang},
  {The even Orlicz Minkowski problem}, Adv. Math. \textbf{224} (2010), no.~6,
  2485--2510.

\bibitem{HP14b}
\textit{C.~Haberl} and \textit{L.~Parapatits}, {The centro-affine Hadwiger
  theorem}, J. Amer. Math. Soc. \textbf{27} (2014), no.~3, 685--705.

\bibitem{HP14a}
\textit{C.~Haberl} and \textit{L.~Parapatits}, Valuations and surface area
  measures, J. Reine Angew. Math. \textbf{687} (2014), 225--245.

\bibitem{HS09b}
\textit{C.~Haberl} and \textit{F.~E. Schuster}, {Asymmetric affine~$L_p$
  Sobolev inequalities}, J. Funct. Anal. \textbf{257} (2009), no.~3, 641--658.

\bibitem{HS09a}
\textit{C.~Haberl} and \textit{F.~E. Schuster}, {General $L_p$ affine
  isoperimetric inequalities}, J. Differential Geom. \textbf{83} (2009), no.~1,
  1--26.

\bibitem{HuaH12}
\textit{Q.~Huang} and \textit{B.~He}, {On the Orlicz Minkowski problem for
  polytopes}, Discrete Comput. Geom. \textbf{48} (2012), no.~2, 281--297.

\bibitem{Kla95}
\textit{D.~Klain}, {A short proof of Hadwiger's characterization theorem},
  Mathematika \textbf{42} (1995), no.~2, 329--339.

\bibitem{LL2016LpMV}
\textit{J.~Li} and \textit{G.~Leng}, {$L_p$ Minkowski valuations on polytopes},
  Adv. Math. \textbf{299} (2016), 139--173.

\bibitem{LL2017OV}
\textit{J.~Li} and \textit{G.~Leng}, Orlicz valuations, Indiana Univ. Math. J.
  \textbf{66} (2017), 791--819.

\bibitem{LM2017Lap}
\textit{J.~Li} and \textit{D.~Ma}, {Laplace transforms and valuations}, J.
  Funct. Anal. \textbf{272} (2017), no.~2, 738--758.

\bibitem{Lin2017OPS}
\textit{Y.~Lin}, {Affine Orlicz P{\'o}lya--Szeg{\"o} principle for log-concave
  functions}, J. Funct. Anal. \textbf{273} (2017), no.~10, 3295--3326.

\bibitem{Lud02b}
\textit{M.~Ludwig}, Projection bodies and valuations, Adv. Math. \textbf{172}
  (2002), no.~2, 158--168.

\bibitem{Lud03}
\textit{M.~Ludwig}, Ellipsoids and matrix-valued valuations, Duke Math. J.
  \textbf{119} (2003), no.~1, 159--188.

\bibitem{Lud05}
\textit{M.~Ludwig}, Minkowski valuations, Trans. Amer. Math. Soc. \textbf{357}
  (2005), no.~10, 4191--4213.

\bibitem{Lud06}
\textit{M.~Ludwig}, Intersection bodies and valuations, Amer. J. Math.
  \textbf{128} (2006), no.~6, 1409--1428.

\bibitem{Lud10a}
\textit{M.~Ludwig}, General affine surface areas, Adv. Math. \textbf{224}
  (2010), no.~6, 2346--2360.

\bibitem{Lud12}
\textit{M.~Ludwig}, {Valuations on Sobolev spaces}, Amer. J. Math. \textbf{134}
  (2012), 824--842.

\bibitem{LR10}
\textit{M.~Ludwig} and \textit{M.~Reitzner}, {A classification of SL$(n)$
  invariant valuations}, Ann. of Math. \textbf{172} (2010), no.~2, 1223--1271.

\bibitem{LR2017sl}
\textit{M.~Ludwig} and \textit{M.~Reitzner}, {SL$(n)$ invariant valuations on
  polytopes}, Discrete Comput. Geom. \textbf{57} (2017), no.~3, 571--581.

\bibitem{Lut93}
\textit{E.~Lutwak}, {The Brunn-Minkowski-Firey theory I: Mixed volumes and the
  Minkowski problem}, J. Differential Geom. \textbf{38} (1993), no.~1,
  131--150.

\bibitem{LYZ00a}
\textit{E.~Lutwak}, \textit{D.~Yang} and \textit{G.~Zhang}, {$L_p$ affine
  isoperimetric inequalities}, J. Differential Geom. \textbf{56} (2000), no.~1,
  111--132.

\bibitem{LYZ02a}
\textit{E.~Lutwak}, \textit{D.~Yang} and \textit{G.~Zhang}, {Sharp affine $L_p$
  Sobolev inequalities}, J. Differential Geom. \textbf{62} (2002), no.~1,
  17--38.

\bibitem{LYZ10b}
\textit{E.~Lutwak}, \textit{D.~Yang} and \textit{G.~Zhang}, Orlicz centroid
  bodies, J. Differential Geom. \textbf{84} (2010), no.~2, 365--387.

\bibitem{LYZ10a}
\textit{E.~Lutwak}, \textit{D.~Yang} and \textit{G.~Zhang}, Orlicz projection
  bodies, Adv. Math. \textbf{223} (2010), no.~1, 220--242.

\bibitem{MS11}
\textit{V.~D. Milman} and \textit{R.~Schneider}, Characterizing the mixed
  volume, Adv. Geom. \textbf{11} (2011), no.~4, 669--689.

\bibitem{Mun2000topo}
\textit{J.~R. Munkres}, Topology, Prentice Hall, 2000.

\bibitem{Par14a}
\textit{L.~Parapatits}, {SL(n)-contravariant $L_p$-Minkowski valuations},
  Trans. Amer. Math. Soc. \textbf{366} (2014), no.~3, 1195--1211.

\bibitem{Par14b}
\textit{L.~Parapatits}, {SL(n)-covariant $L_p$-Minkowski valuations}, J. London
  Math. Soc. \textbf{89} (2014), no.~2, 397--414.

\bibitem{Pet71}
\textit{C.~M. Petty}, Isoperimetric problems, in: Proceedings of the Conference
  on Convexity and Combinatorial Geometry, Department of Mathematics,
  University of Oklahoma, 1971, 26--41.

\bibitem{Rub1998inv}
\textit{B.~Rubin}, {Inversion of fractional integrals related to the spherical
  Radon transform}, J. Funct. Anal. \textbf{157} (1998), no.~2, 470--487.

\bibitem{Sch14}
\textit{R.~Schneider}, {Convex Bodies: The Brunn-Minkowski Theory}, Cambridge
  University Press, Cambridge, 2014, 2nd edition.

\bibitem{SS06}
\textit{R.~Schneider} and \textit{F.~E. Schuster}, {Rotation equivariant
  Minkowski valuations}, Int. Math. Res. Not. \textbf{Art. ID 72894} (2006),
  1--20.

\bibitem{Sch2010}
\textit{F.~E. Schuster}, {Crofton measures and Minkowski valuations}, Duke
  Math. J. \textbf{154} (2010), no.~1, 1--30.

\bibitem{SW2015mink}
\textit{F.~E. Schuster} and \textit{T.~Wannerer}, Minkowski valuations and
  generalized valuations, J. Eur. Math. Soc., in press.

\bibitem{SW12}
\textit{F.~E. Schuster} and \textit{T.~Wannerer}, {$GL(n)$ contravariant
  Minkowski valuations}, Trans. Amer. Math. Soc. \textbf{364} (2012), no.~2,
  815--826.

\bibitem{Tsa12}
\textit{A.~Tsang}, {Minkowski valuations on $L^p$-spaces}, Trans. Amer. Math.
  Soc. \textbf{364} (2012), no.~12, 6159--6186.

\bibitem{Wang13PS}
\textit{T.~Wang}, {The affine P{\'o}lya--Szeg{\"o} principle: equality cases
  and stability}, J. Funct. Anal. \textbf{265} (2013), no.~8, 1728--1748.

\bibitem{Wan11}
\textit{T.~Wannerer}, {GL(n) equivariant Minkowski valuations}, Indiana Univ.
  Math. J. \textbf{60} (2011), no.~5, 1655--1672.

\bibitem{WXL2017+}
\textit{Y.~Wu}, \textit{D.~Xi} and \textit{G.~Leng}, {On the discrete Orlicz
  Minkowski problem}, Trans. Amer. Math. Soc., in press.

\bibitem{XJL14}
\textit{D.~Xi}, \textit{H.~Jin} and \textit{G.~Leng}, {The Orlicz
  Brunn-Minkowski inequality}, Adv. Math. \textbf{260} (2014), 350--374.

\bibitem{Zhang91pro}
\textit{G.~Zhang}, Restricted chord projection and affine inequalities, Geom.
  Dedicata \textbf{39} (1991), no.~2, 213--222.

\bibitem{Zha99Sob}
\textit{G.~Zhang}, {The affine Sobolev inequality}, J. Differential Geom.
  \textbf{53} (1999), 183--202.

\bibitem{ZHY017orl}
\textit{B.~Zhu}, \textit{H.~Hong} and \textit{D.~Ye}, {The Orlicz-Petty
  bodies}, Int. Math. Res. Not. (2017), rnx008, 1-48.

\bibitem{Zhu2014log}
\textit{G.~Zhu}, {The logarithmic Minkowski problem for polytopes}, Adv. Math.
  \textbf{262} (2014), 909--931.

\end{thebibliography}
\end{document}